\documentclass[10pt,twoside,a4paper,reqno]{amsart}

\usepackage{amsmath,amssymb,eucal,graphicx,mathtools}
\usepackage{soul,xcolor}
\usepackage[normalem]{ulem}
\usepackage[foot]{amsaddr}
\usepackage{amsthm}
\usepackage{graphicx}
\usepackage{caption}
\usepackage{subcaption}
\usepackage{listings}
\usepackage{tikz}
\usetikzlibrary{shapes.geometric, positioning,calc, patterns,decorations.markings}
\usetikzlibrary{arrows}
\usetikzlibrary{decorations.pathmorphing}
\usepackage{tikz-cd}
\tikzcdset{arrow style=tikz, diagrams={>=stealth}}
\usepackage{stmaryrd}
\usepackage{comment}
\usepackage{verbatim}
\usepackage{tkz-euclide}
\usepackage{adjustbox}
\usepackage{mathdots}
\usepackage{circuitikz}
\usetikzlibrary{decorations.markings}
\usepackage[noadjust]{cite}
\usepackage{todonotes}
\usepackage{mathrsfs}
\usepackage{mathtools}
\usepackage{enumitem}
\usepackage{microtype}
\usepackage{multicol}
\usepackage[makeroom]{cancel}
\usepackage[margin=2.5cm]{geometry}

\usepackage{hyperref}
\usepackage[capitalise]{cleveref}
\usepackage{thmtools}

\definecolor{friendly_deepblue}{RGB}{51,34,206}
\definecolor{friendly_blue}{RGB}{136,204,238}
\definecolor{friendly_beige}{RGB}{221,104,30}
\definecolor{friendly_red}{RGB}{136,34,85}
\definecolor{friendly_green}{RGB}{68,170,153}

\mathchardef\mhyphen="2D 
\DeclareMathOperator{\mods}{mod}
\newcommand\leftmod[1]{{#1}\mhyphen\!\mods}
\DeclareMathOperator{\Hom}{Hom}
\DeclareMathOperator{\Ext}{Ext}

\DeclareMathOperator{\End}{End}

\DeclareMathOperator{\Aus}{Aus}
\newcommand{\MagnitudeOfModuleCat}[1]{\chi(\leftmod{#1})}

\newcommand{\Z}{\mathbb{Z}}

\counterwithin*{equation}{section}

\newtheorem{theorem}{Theorem}[section]
\newtheorem*{theorem*}{Theorem}
\newtheorem{corollary}[theorem]{Corollary}
\newtheorem{lemma}[theorem]{Lemma}
\newtheorem{proposition}[theorem]{Proposition}

\newtheorem*{conjecture*}{Conjecture}
\newtheorem{mainconjecture}{Conjecture}

\theoremstyle{definition}

\newtheorem{propositiondefinition}[theorem]{Proposition-Definition}
\newtheorem{definition}[theorem]{Definition}

\newtheorem{example}[theorem]{Example}
\newtheorem{remark}[theorem]{Remark}

\newcommand\str{\bgroup\markoverwith{\textcolor{red}{\rule[0.5ex]{2pt}{1.5pt}}}\ULon}

\usepackage{tcolorbox}

\title{Magnitude of module categories}

\author[E. D. Børve]{Erlend D. Børve}
\address[E.D.B]{Institut for Matematik, Aarhus Universitet, Ny Munkegade 118, 8000 Aarhus C, Denmark}
\email{erlend.d.borve@math.au.dk}

\author[D. Horiatakis]{Daniel Horiatakis}
\email{daniel.horiatakis@uni-graz.at}
\author[M. Kalck]{Martin Kalck}
\email[M. Kalck]{martin.kalck@uni-graz.at}
\address[D.H and M.K]{Institut für Mathematik und Wissenschaftliches Rechnen, Universität Graz, Heinrichstraße 36, 8010 Graz, Austria}

\thanks{All authors were partially funded by the Deutsche Forschungsgemeinschaft (DFG, German Research Foundation) -- Projektnummer 496500943. E.D.B. was also funded by the Aarhus University Research Foundation -- Grant number AUFF-E-2024-9-43.}

\keywords{Magnitude, Auslander algebra, translation quiver, Auslander--Reiten quiver, biserial algebra, path algebra, radical square zero algebra, self-injective algebra}
\subjclass[2020]{16G10, 16G70}

\begin{document}

\begin{abstract}
    We define an invariant of the module category of a representation-finite algebra by the magnitude of its Auslander algebra. This invariant will be called the \textit{magnitude of the module category.}
    For bound path algebras, it can be computed as the Euler characteristic of the Auslander--Reiten quiver, in a suitable sense. To aid the computation of our invariant, we define the \textit{Auslander--Reiten--Euler characteristic} of a translation quiver. We build on classical results in Auslander--Reiten theory to determine the magnitude of module categories of biserial algebras, hereditary path algebras, radical square zero bound path algebras, and self-injective bound path algebras. In these cases, we express our invariant in terms of other known quantities, notably the rank of the Grothendieck group and Coxeter numbers of Dynkin quivers. Based on our calculations and results, we obtain a conjectural characterisation of representation-finite biserial algebras in terms of the magnitude of the module category and the rank of the Grothendieck group.
\end{abstract}

\dedicatory{In memory of Idun Reiten.}
\maketitle
\tableofcontents

\section{Introduction}

Over the last two decades, a notion of \textit{magnitude} has been introduced and studied in various settings. It generalises the Euler characteristic of topological spaces to a wide range of other mathematical objects, like metric spaces, graphs, groupoids, and orbifolds \cite{Lei08,Lei13}.
Before it was mathematically developed, the magnitude of metric spaces appeared implicitly in the study of ecology and diversity \cite{SP94}. A mathematically rigorous definition \cite{Lei08,Lei13} later paved the way for interdisciplinary insights between mathematics and ecology \cite{Lei09,LC12,Mec15,LM16,Lei21}.

The most general setting where one considers magnitude is in the framework of enriched categories with an appropriate notion of ``size'' of objects \cite{Lei13, NT16}. In familiar instances, it is clear that one obtains very natural invariants. Indeed, for finite metric spaces, the magnitude is to be regarded as the ``effective number of points'' \cite[§2.1]{Lei13}. For finite skeletal categories whose classifying space is a finite simplicial complex, one recovers the Euler characteristic of the classifying space \cite[Proposition 2.11]{Lei08}. Another suitable setting is formed by Hom-finite linear categories, which lead Chuang--King--Leinster to define the magnitude of a finite-dimensional algebra \cite{CKL16}. It is analogous to the Euler characteristic of topological spaces (see \Cref{prop:CKL+Bongartz}).

The purpose of this text is to develop a theory of magnitude for categories of finitely generated modules over finite-dimensional algebras. At present, our constructions work for \textit{representation-finite algebras}, i.e.\! finite-dimensional algebras admitting only finitely many indecomposable modules up to isomorphism. Given a representation-finite algebra $\Lambda$, we will achieve our aim by investigating the magnitude of its Auslander algebra. This quantity will be called the \textit{magnitude of the module category of $\Lambda$}, and will be denoted $\MagnitudeOfModuleCat{\Lambda}$. It only depends on the category of indecomposable left $\Lambda$-modules (see \Cref{rem:chi_indmod}\eqref{rem:chi_indmod1}). If $\Lambda$ is a bound path $k$-algebra, \Cref{lem:MagARQuiver} lets us compute our invariant as $N - A + E$, where the terms are the numbers of vertices, arrows, and meshes, respectively, in the Auslander--Reiten quiver of $\leftmod{\Lambda}$. Here, $\leftmod{\Lambda}$ denotes the category of finitely generated left $\Lambda$-modules.

Our main results determine the magnitude of the module categories of several important classes of representation-finite algebras $\Lambda$. Recall that the \textit{rank} of $\Lambda$ is defined as the rank of the Grothendieck group $K_0(\leftmod{\Lambda})$. The rank of $\Lambda$ is equal to the number of finitely generated indecomposable projective left $\Lambda$-modules up to isomorphism.
Below, $\Lambda$ is a bound path algebra, so the rank is also equal to the number of vertices in the Gabriel quiver of $\Lambda$.

\begin{theorem*}[\Cref{thm:SpecialBiserial}, \Cref{cor:Nakayama}, and \Cref{cor:Local}]
    Let $\Lambda$ be a representation-finite (special) biserial bound path algebra of rank $n$. Then $\MagnitudeOfModuleCat{\Lambda} = n$. In particular, this holds when $\Lambda$ is a string algebra, a Nakayama algebra, or a representation-finite local bound path algebra.
\end{theorem*}
\begin{theorem*}[\Cref{thm:Dynkin} and \Cref{cor:radsq0}]
    Let {$k$ be a field and let} $Q$ be a Dynkin quiver. Then $\MagnitudeOfModuleCat{kQ} = h_{Q}-1$, where $h_{Q}$ is the Coxeter number of (the underlying Dynkin diagram of) $Q$, see \Cref{propdef:Coxeter}. 
    
    As a consequence, if $\Lambda$ is a radical square zero bound path $k$-algebra of rank $n$, we have that
    $$
    \MagnitudeOfModuleCat{\Lambda} = n+ \sum_{C} (h_C-|C_0|-1),
    $$
    where the sum is indexed over the connected components $C=(C_0,C_1)$ of the separating quiver of $\Lambda$.
\end{theorem*}
\begin{theorem*}[\Cref{thm:Self-injective}]
     Let $\Lambda$ be a connected non-semisimple representation-finite self-injective bound path algebra of rank $n$. 
    Then,
    \begin{align*}
   \MagnitudeOfModuleCat{\Lambda} = n\left(2{h_{T}-1\over |{T}_0| }-1\right),
    \end{align*}
  where $T=(T_0,T_1)$ is a Dynkin quiver such that the stable Auslander--Reiten quiver of $\leftmod{\Lambda}$ is an admissible quotient of the repetitive quiver $\Z T$, and $h_{T}$ is the Coxeter number of (the underlying Dynkin diagram of) $T$. Note that the statement does not depend on the orientation of $T$.
\end{theorem*}

All of our results give evidence to the following:

\begin{mainconjecture}\label{conjecture}
    Let $\Lambda=kQ/I$ be a representation-finite bound path algebra of rank $n$.
    \begin{enumerate}
        \item We have $\MagnitudeOfModuleCat{\Lambda}\geq n$.
        \item We have $\MagnitudeOfModuleCat{\Lambda} = n$ if and only if $\Lambda$ is special biserial.
    \end{enumerate}
\end{mainconjecture} 

We show in \Cref{formula with 1-meshes and 3-meshes} that 
	\begin{equation}\label{eq:introE}
\MagnitudeOfModuleCat{\Lambda}  = n - \ell + |\mathcal{E}_1(\Lambda)| - |\mathcal{E}_3(\Lambda)| - 2|\mathcal{E}_4(\Lambda)|.
	\end{equation} 
    Here, $\ell$ is the number of arrows into projective vertices of the Auslander--Reiten quiver, and $\mathcal{E}_m(\Lambda)$ denotes the number of isomorphism classes of Auslander--Reiten sequences in $\leftmod{\Lambda}$ with $m$ indecomposable middle terms. One possible approach to proving \Cref{conjecture} could be to compare the terms on the right hand side of \eqref{eq:introE} with each other.

\textbf{Structure of the paper.} 
We begin by arguing that the magnitude of an Auslander algebra is an invariant under Morita equivalence. It is not, however, invariant under derived Morita equivalence, see \Cref{rem:tilting}. The rest of \Cref{sec: Magnitude of Auslander algebras} is dedicated to developing tools which will be used to calculate magnitudes of specific Auslander algebras in later sections. 

In \Cref{sec: Translation quivers}, we recall the notion of a translation quiver, and define the Auslander--Reiten--Euler characteristic of such. One specific type of translation quivers are repetitive quivers. Since many Auslander--Reiten quivers are closely related to them, they provide a useful tool for magnitude computations in later sections. In Sections \ref{sec:biserial}, \ref{sec:Dynkin}, and \ref{sec:SelfInj}, we prove our main results listed above.  

\textbf{Notation and conventions.} Throughout the text, we fix a field $k$. Given a finite-dimensional $k$-algebra $\Lambda$, the word \textit{module} will mean finitely generated left $\Lambda$-module. 
If $a$ and $b$ are arrows in a quiver, such that the target of $a$ is the source of $b$, we denote the concatenated path by $ba$.

\textbf{Acknowledgements.} 
The authors thank Gustavo Jasso, Alastair King, David Pauksztello, and Gordana Todorov for useful discussions.
E.D.B. is grateful to M.K. for organising a one-week research visit in Graz in February 2026.

\section{Magnitude of Auslander algebras}\label{sec: Magnitude of Auslander algebras}

\subsection{Magnitude of finite-dimensional algebras}

\begin{definition}[{\cite{Lei08}}]\label{def:Lei08}
    Fix a ring $(R,+,\cdot)$ with multiplicative identity element $1_{R}$. Let $n\geq 1$, let $\mathbf{Z}$ be an $n\times n$ square matrix over $R$, and let $\mathbf{1}$ denote the $n \times 1$ column vector in which each entry is $1_{R}$. 
    \begin{enumerate}
        \item A \textit{weighting} of $\mathbf{Z}$ is a column vector $\mathbf{w}$ over $R$ such that $\mathbf{Z}\mathbf{w} = \mathbf{1}$. Dually, a \textit{co-weighting} of $\mathbf{Z}$ is a row vector $\mathbf{c}$ over $R$ such that $\mathbf{c}\mathbf{Z} = \mathbf{1}^T$.
        \item We say that $\mathbf{Z}$ \textit{has magnitude over $R$} if it admits at least one weighting and at least one co-weighting.
        In this case, let $\mathbf{w} =(w_i)_i$ be a weighting of $\mathbf{Z}$, and let $\mathbf{c}=(c_i)_i$ be a co-weighting of $\mathbf{Z}$. The \textit{magnitude} of $\mathbf{Z}$ is defined as follows:
        \begin{equation*}
            \chi(\mathbf{Z}) \coloneqq \mathbf{c} \mathbf{Z} \mathbf{w} \in R.
        \end{equation*}
        \end{enumerate}
\end{definition}

By the associativity of matrix multiplication, the following equalities hold: 
       \begin{equation*}
            \chi(\mathbf{Z}) = \sum_i w_i = \sum_i c_i.
        \end{equation*}
    Consequently, the magnitude of $\mathbf{Z}$ is independent of the choice of weighting and co-weighting.

\begin{definition}[{\cite[Example 1.2.2(ii)]{Lei08},\cite{CKL16}}]\label{def:CKL16_magn}
\,
    \begin{enumerate}
        \item\label{def:CKL16_magn_1} Let $n\geq 1$ and let $\mathcal{C}$ be a $k$-linear category with $n$ objects up to isomorphism such that $\Hom_{\mathcal{C}}(X,Y)$ is finite-dimensional for all $X,Y\in\mathcal{C}$. The \textit{similarity matrix} of $\mathcal{C}$ is the $n\times n$-matrix $\mathbf{Z}_{\mathcal{C}}\coloneqq (\dim_k\Hom_{\mathcal{C}}(X_i,X_j))_{i,j}$, where $\{X_i\}_{i=1}^n$ is a complete set of pairwise non-isomorphic objects in $\mathcal{C}$.
        We say that $\mathcal{C}$ \textit{has magnitude} if $\mathbf{Z}_{\mathcal{C}}$ has magnitude as a matrix over $\mathbb{Q}$.
        The \textit{magnitude} of $\mathcal{C}$ is then defined as the magnitude of $\mathbf{Z}_{\mathcal{C}}$.
        \item Let $\Gamma$ be a finite-dimensional $k$-algebra. We define the \textit{similiarity matrix} of $\Gamma$ as the similarity matrix of the category of finitely generated indecomposable projective left $\Gamma$-modules. We denote it by $\mathbf{Z}_{\Gamma}$.
        We say that $\Gamma$ \textit{has magnitude} if $\mathbf{Z}_{\Gamma}$ has magnitude as a matrix over $\mathbb{Q}$. The \textit{magnitude} of $\Gamma$ is then the rational number $\chi(\mathbf{Z}_{\Gamma})$, henceforth denoted $\chi(\Gamma)$.
    \end{enumerate}
\end{definition}

If $k$ is algebraically closed, then a semisimple $k$-algebra of rank $n$ has magnitude $n$, since its similarity matrix is the identity $n\times n$-matrix. As a less trivial example, consider the complex numbers as a two-dimensional real algebra. Its similarity matrix is the $1\times 1$-matrix $\begin{bmatrix} 2 \end{bmatrix}$, yielding that the magnitude is equal to $\frac{1}{2}$.

In \Cref{def:CKL16_magn}\eqref{def:CKL16_magn_1}, the similarity matrix of $\mathcal{C}$ is independent of the choice of the set $\{X_i\}_{i=1}^n$. Moreover, different orderings of $\{X_i\}_{i=1}^n$ determine the same similarity matrices up to a permutation of columns. Thus, the ordering will neither affect whether $\mathcal{C}$ has magnitude, nor the value of $\chi(\mathcal{C})$ when it is defined.

Two finite-dimensional $k$-algebras $\Gamma_1$ and $\Gamma_2$ are \textit{Morita equivalent} if there exists an exact equivalence
\begin{equation*}
    \begin{tikzcd}
        \leftmod{\Gamma_1} \arrow[r] & \leftmod{\Gamma_2} 
    \end{tikzcd}
\end{equation*}
between categories of finitely generated left modules. Morita equivalence only alters the similarity matrix up to a permutation of columns. Hence, if $\Gamma_1$ and $\Gamma_2$ are Morita equivalent, then $\Gamma_1$ has magnitude precisely when $\Gamma_2$ has magnitude. In this case, we have $\chi(\Gamma_1)=\chi(\Gamma_2)$. In other words, the magnitude of finite-dimensional $k$-algebras is invariant under Morita equivalence.

\begin{remark}\label{rem:Cartan}
    Let $n\geq 1$. 
    \begin{enumerate}
        \item\label{rem:Cartan1} If $\mathbf{Z}$ is an invertible $n\times n$-matrix over $R$, then it admits a unique weighting $\mathbf{Z}^{-1}\mathbf{1}$ and a unique co-weighting $\mathbf{1}^T\mathbf{Z}^{-1}$. In particular, it has magnitude over $R$, which, by definition, is equal to the sum of the entries in $\mathbf{Z}^{-1}$ (see \cite[Lemma 1.1.4]{Lei13}).
        \item\label{rem:Cartan1.5} Let $\Gamma$ be a finite-dimensional $k$-algebra and let $\{P_i\}_{1 \leq i \leq n}$ (resp.\! $\{S_i\}_{1 \leq i \leq n}$) denote the set of finitely generated indecomposable projective (resp. simple) left $\Gamma$-modules up to isomorphism. One defines the \textit{Cartan matrix} $\Gamma$ as the $n\times n$-matrix $\mathbf{C}_{\Gamma}=(c_{i,j})_{i,j}$, where $c_{i,j}$ is the multiplicity of $S_i$ as a composition factor of $P_j$. Suppose in addition that $\Gamma$ has finite global dimension. Since $\det(\mathbf{C}_{\Gamma})$ then is equal to $\pm 1$ \cite[Proposition 21]{Eil54} (conjecturally only $1$), it follows from \eqref{rem:Cartan1} that $\mathbf{C}_\Gamma$ has magnitude over $\Z$.
        \item\label{rem:Cartan2} 
        Let $\mathbf{Z}_{\Gamma}$ denote the similarity matrix of $\Gamma$ and let $d_i$  denote $\dim_k\End_{\Gamma}(S_i)$ for $1\leq i\leq n$. Then $(\mathbf{Z}_{\Gamma})_{i,j} = d_i c_{i,j}$ (see \cite[§4]{CKL16}).
        It now follows from \eqref{rem:Cartan1.5} that $\Gamma$ has magnitude whenever it has finite global dimension.
    \end{enumerate}
\end{remark}

In this text, we focus our attention on finite-dimensional $k$-algebras that are of the form $kQ/I$, where $Q$ is a finite quiver and $I$ is an admissible ideal of the path $k$-algebra $kQ$. These will be referred to as \textit{bound path $k$-algebras}.
We refer to $Q$ as the \textit{Gabriel quiver} of $\Lambda$, as it is uniquely determined up to isomorphism of quivers (see \cite[Theorem III.1.9(d)]{ARS97}). Since all simple modules are one-dimensional in this case, the similarity matrix is equal to the Cartan matrix. {By \Cref{rem:Cartan}\eqref{rem:Cartan1.5},} we have that $\chi(kQ/I)$ is an integer if $kQ/I$ has finite global dimension.
Note that if $k$ is algebraically closed, all finite-dimensional $k$-algebras are bound path $k$-algebras up to Morita equivalence.

\begin{proposition}[{\cite[Theorem 1.1]{CKL16}}]\label{prop:CKL+Bongartz}
    Let $\Gamma = kQ/I$ be a bound path $k$-algebra and let $S$ be the semisimple $\Gamma$-module $\Gamma / \mathrm{rad}(\Gamma)$, where $\mathrm{rad}(\Gamma)$ denotes the Jacobson radical of $\Gamma$.
    Suppose that $\Gamma$ has finite global dimension $d$. Then the magnitude of $\Gamma$ is given by 
        \begin{equation}\label{eq:CKL}
            \chi(\Gamma) = \sum_{i=0}^d (-1)^i \dim_k \Ext^i_{\Gamma}(S,S).
        \end{equation}
\end{proposition}

\begin{remark}\label{rem:CKL_fieldext}
    Let $\Gamma$ and $S$ be as in \Cref{prop:CKL+Bongartz}, and let $K$ be a field extension of $k$.
    \begin{enumerate}  
    \item\label{rem:CKL_fieldext1} The \textit{extension of scalars} of $\Gamma$ by $K$ is defined as the $K$-algebra $\Gamma^K \coloneqq K\otimes_k \Gamma$. 
    It is a bound path $K$-algebra, since it is isomorphic to $KQ/(K\otimes_k I)$. In particular, the Gabriel quiver of the $K$-algebra $\Gamma^K$ is isomorphic to that of the $k$-algebra $\Gamma$. Since both $\Gamma$ and $\Gamma^K$ are bound path algebras, the natural homomorphism $\Gamma^K/\mathrm{rad}(\Gamma^K) \to K \otimes_k S$
    is an isomorphism of semisimple $\Gamma^K$-modules.
    \item\label{rem:CKL_fieldext3} For any $i\geq 0$, there is a natural isomorphism of bifunctors
    \begin{equation}\label{eq:KasjanExt}
        \begin{tikzcd}
        K\otimes_k \Ext^i_{\Gamma}(-,-) \arrow[r] & \Ext^i_{\Gamma^K}(K\otimes_k -,K\otimes_k -)
   	 \end{tikzcd}
    \end{equation}
     from $(\leftmod{\Gamma})^{\mathrm{op}}\times \leftmod{\Gamma}$ to $\leftmod{K}$ { \cite[Lemma 2.2(a)]{Kas00}}. By \eqref{rem:CKL_fieldext1} and \Cref{prop:CKL+Bongartz}, it follows that $\chi(\Gamma^K)=\chi(\Gamma)$.
    \end{enumerate}
\end{remark}

\subsection{Auslander algebras}
    
Let $\Lambda$ be a representation-finite $k$-algebra. A \textit{representation-generator} of $\Lambda$ is a $\Lambda$-module $M$ such that each indecomposable $\Lambda$-module is a direct summand of $M$. 
The endomorphism $k$-algebra $\mathrm{End}_{\Lambda}(M)$ is called the \textit{Auslander algebra of $\Lambda$}. Although it depends on the choice of the representation-generator, it is uniquely determined up to Morita equivalence. We will denote the (basic) Auslander algebra of $\Lambda$ by $\mathrm{Aus}(\Lambda)$. Auslander showed that $\mathrm{Aus}(\Lambda)$ has global dimension at most 2 and dominant dimension at least 2 \cite{Aus71}. A finite-dimensional $k$-algebra with these properties is called an \textit{Auslander $k$-algebra}.
The \textit{Auslander correspondence} \cite{Aus71} is given by the bijection
\begin{equation*}
    \begin{tikzcd}[row sep=1em]
        \big\{\text{representation-finite $k$-algebras} \big\}/\text{Morita equiv.}\arrow[r] & \big\{\text{Auslander $k$-algebras}\big\}/\text{Morita equiv.} \\
        {[\Lambda]} \arrow[r,mapsto] \arrow[u,phantom,sloped,"\in"] &  \arrow[u,phantom,sloped,"\in"] [\mathrm{Aus}(\Lambda)],
    \end{tikzcd}
\end{equation*}
where $[-]$ denotes Morita equivalence classes. We are now ready to present the main object of study in this paper.

\begin{definition}\label{def:MagnitudeOfModuleCat}
    Let $\Lambda$ be a representation-finite $k$-algebra. We define the \textit{magnitude of the module category of $\Lambda$} to be $\chi(\Aus(\Lambda))$, namely the magnitude of the Auslander algebra of $\Lambda$. We denote this quantity by $\MagnitudeOfModuleCat{\Lambda}$.
\end{definition}

The Auslander correspondence preserves Morita equivalences. Thus, Morita equivalent representation-finite $k$-algebras have module categories of the same magnitude. \Cref{def:MagnitudeOfModuleCat} thus provides a map
\begin{equation*}
    \begin{tikzcd}[row sep=1em]
        \big\{\text{representation-finite $k$-algebras} \big\}/\text{Morita equiv.}\arrow[r,"\MagnitudeOfModuleCat{-}"] & \mathbb{Q} \\
        {[\Lambda]} \arrow[r,mapsto] \arrow[u,phantom,sloped,"\in"] &  \arrow[u,phantom,sloped,"\in"] \MagnitudeOfModuleCat{\Lambda}
    \end{tikzcd}
\end{equation*}
Since Auslander algebras have finite global dimension, this is well-defined as a result of {\Cref{rem:Cartan}\eqref{rem:Cartan2}}. By \Cref{prop:CKL+Bongartz}, this map returns an integer whenever the input is Morita equivalent to a bound path $k$-algebra.

\begin{remark}\label{rem:chi_indmod}
\, \\
Let $\Lambda$ be a representation-finite $k$-algebra, and fix a basic representation-generator $M=\bigoplus\limits_{i=1}^N M_i$ of $\Lambda$.
\begin{enumerate}
    \item\label{rem:chi_indmod1} By projectivisation (see \cite[§II.2]{ARS97}), the functor
    \[\begin{tikzcd}
          \Hom_{\Lambda}(-,M)\colon  \mathrm{ind}(\leftmod{\Lambda})^{\mathrm{op}} \arrow[r] & \mathrm{ind}(\mathrm{Aus}(\Lambda)\mhyphen\mathrm{proj})
    \end{tikzcd}\]
    is an equivalence of $k$-linear categories. It goes from the opposite category of indecomposable left $\Lambda$-modules to the category of indecomposable projective left $\mathrm{Aus}(\Lambda)$-modules. Under this equivalence, we may identify the similarity matrix of $\mathrm{Aus}(\Lambda)$ with $(\dim_k\Hom_{\Lambda}(M_j, M_i))_{i,j}$. 
    The similarity matrix of $\mathrm{ind}(\leftmod{\Lambda})$ then becomes the transpose of the similarity matrix of $\mathrm{Aus}(\Lambda)$. We deduce that
    \begin{equation}\label{eq:chiindmod}
       \MagnitudeOfModuleCat{\Lambda} = \chi(\mathrm{ind}(\leftmod{\Lambda})),
    \end{equation}
    where the right hand side denotes the magnitude of $\mathrm{ind}(\leftmod{\Lambda})$ as a $k$-linear category with finitely many objects up to isomorphism. In particular, the quantity $\MagnitudeOfModuleCat{\Lambda}$ only depends on $\mathrm{ind}(\leftmod{\Lambda})$.
    \item\label{Splitting formula}
    Suppose that $\Lambda$ is a direct product $\Lambda_1\times \Lambda_2$. Since $\mathrm{ind}(\leftmod{\Lambda})$ is then equivalent to the disjoint union $\mathrm{ind}(\leftmod{\Lambda_1})\sqcup \mathrm{ind}(\leftmod{\Lambda_2})$, the similarity matrix of $\mathrm{ind}(\leftmod{\Lambda})$ can be made block diagonal. By \eqref{rem:chi_indmod1}, we find that 
    \begin{equation*}
        \MagnitudeOfModuleCat{\Lambda}  = \MagnitudeOfModuleCat{\Lambda_1} + \MagnitudeOfModuleCat{\Lambda_2}.
    \end{equation*}
\end{enumerate}
\end{remark}

Let $\Lambda$ be a bound path $k$-algebra.
Recall that the \textit{Auslander--Reiten quiver} of $\leftmod{\Lambda}$ is the quiver $\mathfrak{AR}^{\Lambda}$ whose set of vertices is given by isomorphism classes of finitely generated indecomposable left $\Lambda$-modules, and arrows are given by irreducible morphisms. Let $\tau$ denote the Auslander--Reiten translation in $\leftmod{\Lambda}$. 
In \Cref{sec: Translation quivers}, we include a systematic treatment of translation quivers, of which the pair $(\mathfrak{AR}^{\Lambda},\tau)$ forms an example.

\begin{remark}\label{rem:fielddoesntmatter}
    Let $\Lambda=kQ/I$ be a representation-finite bound path $k$-algebra.
    \begin{enumerate}
    \item\label{rem:fdnm_AuslanderGabriel} We have that $\mathrm{Aus}(\Lambda)$ is Morita equivalent to a bound path $k$-algebra. Its Gabriel quiver is isomorphic to $\mathfrak{AR}^{\Lambda}$ (see \cite[Theorem VII.1.6]{ARS97}).
    \item\label{rem:fdnm_magnitude} Let $M$ be a representation-generator of $\Lambda$ and let $K$ be a field extension of $k$. One applies \eqref{eq:KasjanExt} to show that \[K\otimes_k \End_{\Lambda}(M) \simeq \End_{\Lambda^K}(K\otimes_k M)\] as $K$-algebras. Consequently, we have that $\Aus(\Lambda^K)$ is Morita equivalent to $\Aus(\Lambda)^K$. Since the Auslander--Reiten translation in $\leftmod{\Lambda}$ is compatible with extension of scalars \cite[Corollary 3.6]{Kas00}, it follows from \eqref{rem:fdnm_AuslanderGabriel} and \Cref{rem:CKL_fieldext}\eqref{rem:CKL_fieldext1} that $(\mathfrak{AR}^{\Lambda},\tau)$ and $(\mathfrak{AR}^{\Lambda^K},\tau)$ are isomorphic as translation quivers. Using \Cref{rem:CKL_fieldext}\eqref{rem:CKL_fieldext3}, we also deduce that
    \[\MagnitudeOfModuleCat{\Lambda^K} = \MagnitudeOfModuleCat{\Lambda}.\]
    \end{enumerate}
\end{remark}

There is a one-to-one correspondence between isomorphism classes of Auslander--Reiten sequences in $\leftmod{\Lambda}$ and non-projective indecomposable $\Lambda$-modules. If $\Lambda$ is representation-finite, we obtain
\begin{equation}\label{eq:E}
   E = N-n,
\end{equation}
where $E$ is the number of isomorphism classes of Auslander--Reiten sequences in $\leftmod{\Lambda}$, $N$ is the number of indecomposable $\Lambda$-modules up to isomorphism, and $n$ is the rank of $\Lambda$.

\begin{lemma}\label{lem:MagARQuiver}
    Let $\Lambda$ be a representation-finite $k$-algebra, let $E$, $N$, and $n$ be as above, and let $A$ be the number of arrows in $\mathfrak{AR}^{\Lambda}$. Then we have that
         \begin{equation}\label{eq:MagARQuiver}
    \MagnitudeOfModuleCat{\Lambda} = N - A + E \stackrel{\eqref{eq:E}}{=} 2N - A - n.
\end{equation}
\end{lemma}
\begin{proof}
    Since Auslander algebras have global dimension at most 2, we use \Cref{prop:CKL+Bongartz} to assert that
    \begin{equation}\label{eq:MagARQuiverCKL}
        \MagnitudeOfModuleCat{\Lambda} = \dim_k\Ext^0_{\Aus(\Lambda)}(S,S) - \dim_k\Ext^1_{\Aus(\Lambda)}(S,S) + \dim_k\Ext^2_{\Aus(\Lambda)}(S,S).
    \end{equation}
    Here, $S$ is a basic direct sum of all simple $\Aus(\Lambda)$-modules up to isomorphism.
    If $M$ is a basic representation-generator of $\Lambda$, we may identify $\Aus(\Lambda)$ with $\End_{\Lambda}(M)$ up to Morita equivalence. By \Cref{rem:chi_indmod}\eqref{rem:chi_indmod1}, a complete set of indecomposable projective $\Aus(\Lambda)$-modules up to isomorphism is given by $\{\Hom_{\Lambda}(C,M)\}_C$. This set is indexed over the isomorphism classes of indecomposable $\Lambda$-modules.
    Letting $S_C$ denote the simple top of $\Hom_{\Lambda}(C,M)$, we get a decomposition $S\simeq \bigoplus_C S_C$.
    Since $\mathrm{Aus}(\Lambda)$ is (Morita equivalent to) a bound path $k$-algebra by \Cref{rem:fielddoesntmatter}\eqref{rem:fdnm_AuslanderGabriel}, we have that 
    $\dim_k\Hom_{\mathrm{Aus}(\Lambda)}(S_C,S)
    = 1.$ This shows that \[\dim_k \Ext^0_{\Aus(\Lambda)}(S,S) = \dim_k \Hom_{\Aus(\Lambda)}(S,S) = |\mathfrak{AR}^{\Lambda}_0| = N.\]
     To conclude the proof, we show that
    \begin{equation}\label{eq:MagARcases1} 
    {\dim_k\Ext^1_{\Aus(\Lambda)}(S_C,S) = |\{\text{arrows out of }C\text{ in }\mathfrak{AR}^{\Lambda}\}|,}
    \end{equation}
    {and that}
    \begin{equation} \label{eq:MagARcases2}
        \dim_k\Ext^2_{\Aus(\Lambda)} (S_C,S) = \begin{cases}
            1 & \text{if } C\text{ is non-injective,} \\
            0 & \text{else}.
        \end{cases}
    \end{equation}
    We distinguish between two cases. If $C$ is non-injective, there exists an Auslander--Reiten sequence of the form
    \begin{equation}\label{eq:ARseq}
    \begin{tikzcd}[ampersand replacement=\&]
        0 \arrow[r] \& C \arrow[r, tail,"(g_i)_i"] \& {\bigoplus\limits_{i=1}^m B_i} \arrow[r, two heads,"(f_i)_i"] \& \tau^{-1}{C} \arrow[r] \& {0}.
    \end{tikzcd}
\end{equation}
    A projective resolution of $S_{C}$ is then given by
        \begin{equation}\label{eq:projres}
       \begin{tikzcd}
     0 \arrow[r] & {\Hom_{\Lambda}(\tau^{-1}C,M)} \arrow[r] & {\bigoplus\limits_{i=1}^m \Hom_{\Lambda}(B_i,M)} \arrow[r] & {\Hom_{\Lambda}(C,M)}\arrow[r,two heads] & S_{C}.
\end{tikzcd}
    \end{equation}
    By \Cref{rem:fielddoesntmatter}\eqref{rem:fdnm_AuslanderGabriel}, $\mathrm{Aus}(\Lambda)$ is a bound path $k$-algebra whose Gabriel quiver is the Auslander--Reiten quiver of $\leftmod{\Lambda}$. Since $\mathrm{Aus}(\Lambda)$ has finite global dimension, it follows that the Auslander--Reiten quiver of $\leftmod{\Lambda}$ has no loops \cite[Corollary 5.6]{Igu90}. In particular, consecutive terms in the projective resolution in \eqref{eq:projres} share no indecomposable direct summand. Thus, applying $\Hom_{\mathrm{Aus}(\Lambda)}(-,S)$ induces trivial differentials. Since the $g_i$s are all the arrows out of $C$ in $\mathfrak{AR}^{\Lambda}$, the formula in \eqref{eq:MagARcases1} follows in this case.
    We also readily deduce \eqref{eq:MagARcases2}.
    {In the remaining case, where} $C$ is injective, a projective resolution of $S_{C}$ is given by 
    \begin{equation*}
        \begin{tikzcd}
 0 \arrow[r] & \Hom_{\Lambda}(C/\mathrm{soc}(C),M) \arrow[r] &
  \Hom_{\Lambda}(C,M)\arrow[r,two heads] & S_{C},
\end{tikzcd}
    \end{equation*}
    where $\mathrm{soc}(C)$ denotes the socle of $C$ in $\leftmod{\Lambda}$. Since the number of indecomposable direct summands of $C/\mathrm{soc(C)}$ is the number of arrows out of $C$ in $\mathfrak{AR}^{\Lambda}$ (see \cite[Remark IV.4.3(b)]{ASS06}), we find that \eqref{eq:MagARcases1} also holds in this case. Moreover, we deduce \eqref{eq:MagARcases2} from the fact that  the functor $\Ext^2_{\mathrm{Aus}(\Lambda)}(S_C,-)$ now vanishes.
\end{proof}

We have reduced the computation of $\MagnitudeOfModuleCat{\Lambda}$ to Auslander--Reiten theory. 

\begin{example}\label{ex:A3}
    Let $Q$ be the following quiver:
    \[
    \begin{tikzcd}
        1 \arrow[r,"a"] & 2 \arrow[r,"b"] & 3.
    \end{tikzcd}
    \]
    The Auslander--Reiten quiver of $\leftmod{kQ}$ has the form
   
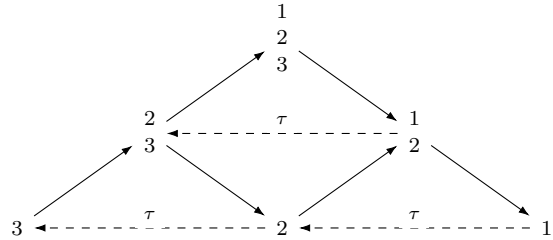
\[ \begin{tikzpicture}[x=0.035cm,y=-0.025cm]
  \clip (0,345) rectangle + (240,140);
  \begin{scope}[every node/.style={inner sep=4pt,font=\footnotesize}]
   \node (1_0) at (120,377) {$\begin{matrix} 1 \\ 2 \\ 3 \end{matrix}$};
   \node (2_0) at (70,427) {$\begin{matrix}  2 \\ 3 \end{matrix}$};
   \node (3_0) at (20,477) {$3$};
   \node (3_1) at (120,477) {$2$};
   \node (2_1) at (170,427) {$\begin{matrix}  1 \\ 2 \end{matrix}$};
   \node (3_2) at (220,477) {$1$};
  \end{scope}
  \begin{scope}[every node/.style={fill=white,font=\scriptsize},every path/.style={-{Latex[length=1.5mm,width=1mm]}}]
   \path (3_0) edge (2_0);
   \path (3_1) edge[dashed, "$\tau$"'] (3_0);
   \path (2_0) edge (3_1);
   \path (2_0) edge (1_0);
   \path (2_1) edge[dashed, "$\tau$"'] (2_0);
   \path (3_1) edge (2_1);
   \path (1_0) edge (2_1);
   \path (3_2) edge[dashed, "$\tau$"'] (3_1);
   \path (2_1) edge (3_2);
  \end{scope}
 \end{tikzpicture}\] 
 where we have labelled the vertices by composition series of indecomposable $kQ$-modules. Using \Cref{lem:MagARQuiver}, one finds that $\MagnitudeOfModuleCat{kQ}=3$.
\end{example}

\begin{remark}
    Given an integer $n\geq 2$, one can find representation-finite $k$-algebras of rank $n$ with a module category of arbitrarily large magnitude in the following way. By \Cref{rem:chi_indmod}\eqref{Splitting formula}, it suffices to give a suitable family of examples for $n=2$, because we may then consider a product with a semisimple $k$-algebra. 
    Let $\omega$ be an even positive integer and let $\Lambda_{\omega}$ denote the bound path $k$-algebra given by the quiver
    
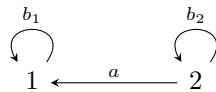
\begin{equation*}
        \begin{tikzcd}
1 \arrow["b_1"', loop, distance=2em, in=125, out=55] &  & 2 \arrow[ll, "a"'] \arrow["b_2"', loop, distance=2em, in=125, out=55]
\end{tikzcd}
    \end{equation*}
    modulo the relations $b_1a - ab_2$, $b_1^{\omega}$, and $b_2^{2}$. The Auslander--Reiten quiver of $\leftmod{\Lambda_\omega}$ has ${(\omega+1)(\omega+4)\over 2}$ vertices and {$\omega(\omega+4)+2$} arrows (see \cite[Example 6 on p. 370]{BG82}). By \Cref{lem:MagARQuiver}, we have $\MagnitudeOfModuleCat{\Lambda_{\omega}} = \omega$. We will return to a discussion of upper bounds in \Cref{prop:UpperBoundT}.
\end{remark}

\subsection{First properties}

In \eqref{eq:ARseq}, we displayed an arbitrary Auslander--Reiten sequence where the middle term is a direct sum of $m$ indecomposable modules. In this case, we say that this Auslander--Reiten sequence \textit{has $m$ middle terms}.
From now on, let $\mathcal{E}_m(\Lambda)$ denote the set of isomorphism classes of Auslander--Reiten sequences in $\leftmod{\Lambda}$ with exactly $m$ middle terms. 

\begin{lemma}[Four in the middle \cite{BB83}]\label{lem:4inthemiddle}
    Let $\Lambda$ be a representation-finite $k$-algebra. We have that $\mathcal{E}_m(\Lambda)$ is empty for $m > 4$. In other words, every Auslander--Reiten sequence in $\leftmod{\Lambda}$ has at most four middle terms.
\end{lemma}

We use \Cref{lem:4inthemiddle} to rewrite our formula for the magnitude of the module category. Recall that a bound path $k$-algebra $\Lambda = kQ/I$ is \textit{monomial} if the ideal $I$ can be generated by paths in $Q$.
\begin{lemma}\label{formula with 1-meshes and 3-meshes}
	Let $\Lambda=kQ/I$ be a representation-finite bound path $k$-algebra of rank $n$. We then have that 
	\begin{equation*}
	\MagnitudeOfModuleCat{\Lambda}  = n - \ell + |\mathcal{E}_1(\Lambda)| - |\mathcal{E}_3(\Lambda)| - 2|\mathcal{E}_4(\Lambda)|,
	\end{equation*}
	where $\ell$ is the number of arrows in the Auslander--Reiten quiver whose codomain is a projective $\Lambda$-module. If $\Lambda$ is monomial, we have $\ell = |Q_1|$. In this case
    \begin{equation*}
	\MagnitudeOfModuleCat{\Lambda} =  n - |Q_1| + |\mathcal{E}_1(\Lambda)| - |\mathcal{E}_3(\Lambda)| -2|\mathcal{E}_4(\Lambda)|.
	\end{equation*}
	\end{lemma}
\begin{proof}
	By \Cref{lem:4inthemiddle}, we have that $\mathcal{E}_m(\Lambda)$ is empty for $m > 4$. Counting the number of arrows into each vertex of the Auslander--Reiten quiver, we find that
    \begin{align*}
    A &= \ell + |\mathcal{E}_1(\Lambda)|  + 2|\mathcal{E}_2(\Lambda)|  + 3|\mathcal{E}_3(\Lambda)| + 4|\mathcal{E}_4(\Lambda)| \\
    &= \ell + 2(|\mathcal{E}_1(\Lambda)| + |\mathcal{E}_2(\Lambda)| + |\mathcal{E}_3(\Lambda)| + |\mathcal{E}_4(\Lambda)|) - |\mathcal{E}_1(\Lambda)| + |\mathcal{E}_3(\Lambda)| + 2|\mathcal{E}_4(\Lambda)|.
    \end{align*}
    By \eqref{eq:E}, the quantity $|\mathcal{E}_1(\Lambda)| + |\mathcal{E}_2(\Lambda)| + |\mathcal{E}_3(\Lambda)| + |\mathcal{E}_4(\Lambda)|=E$ is equal to $N-n$.
    Substitution gives that    
    \begin{equation*}
    A = \ell + 2(N-n) - |\mathcal{E}_1(\Lambda)| + |\mathcal{E}_3(\Lambda)| + 2|\mathcal{E}_4(\Lambda)|.
    \end{equation*}
    Finally, the formula in \Cref{lem:MagARQuiver} gives
    \begin{equation*}
    	\MagnitudeOfModuleCat{\Lambda} = 2N - A - n = n - \ell + |\mathcal{E}_1(\Lambda)| - |\mathcal{E}_3(\Lambda)| - 2|\mathcal{E}_4(\Lambda)|,
    \end{equation*}
    as claimed. 
    
    To prove the remaining assertions, it suffices to show that $\ell = |Q_1|$ when $\Lambda$ is monomial. It is well-known that $\ell$ is equal to the number indecomposable direct summands of $\mathrm{rad}(\Lambda)$ up to isomorphism (see \cite[Remark IV.4.3(a)]{ASS06}). 
    In the following, we will regard $\mathrm{rad}(\Lambda)$ as a $kQ$-module, which will not affect its number of indecomposable direct summands. 
    Since $\Lambda$ is monomial, a $k$-basis of the $kQ$-module $\mathrm{rad}(\Lambda)$ is given by
    \[B\coloneqq \{p \,:\, \text{$p$ path of length $\geq 1$ in $Q$ such that $p\not\in I$} \}.\]
    Let $C$ denote the coefficient quiver of $\mathrm{rad}(\Lambda)$ with respect to $B$ (see \cite{Rin98}).
    The vertices of $C$ are the elements in $B$ and there is an arrow $p \xrightarrow{a} q$ whenever $q=ap$, where $a\in Q_1$. The number of connected components of $C$ is $|Q_1|$. On the other hand, this is also the number of indecomposable direct summands of $\mathrm{rad}(\Lambda)$ \cite[Property 1]{Rin98}.
\end{proof}

We conclude this section with an important result, which is very useful when computing the magnitude of module categories in practice. It also plays a significant role in the proofs of \Cref{thm:SpecialBiserial,thm:Self-injective}.

\begin{lemma}[Drozd--Kiri\v{c}enko Rejection Lemma \cite{DK72}]\label{lem:DK}
    Let $\Lambda=kQ/I$ be a bound path $k$-algebra and let $U$ be a non-simple indecomposable projective-injective module. Then the socle $\mathrm{soc}(U)$ is a two-sided ideal of $\Lambda$ \cite[Lemma 4.3]{AR80}. Let $\mathfrak{AR}^{\Lambda}$ denote the Auslander--Reiten quiver of $\leftmod{\Lambda}$ and let $\mathfrak{AR}^{\Lambda/\mathrm{soc}(U)}$ denote that of the quotient $k$-algebra $\Lambda/\mathrm{soc}(U)$. 
    The following assertions hold.
    \begin{enumerate}
        \item\label{lem:DK_ind} We have 
        $\mathfrak{AR}^{\Lambda/\mathrm{soc}(U)}_0 = \mathfrak{AR}^{\Lambda}_0 \setminus \{U\}$.
        \item\label{lem:DK_arrows} 
        The arrows $\mathrm{rad}(U) \xrightarrow{q_1} U$ and $U \xrightarrow{q_2} U/\mathrm{soc}(U)$ are the only arrows in $\mathfrak{AR}^{\Lambda}$ that are incident to $U$. It follows from \eqref{lem:DK_ind} that $\mathfrak{AR}^{\Lambda/\mathrm{soc}(U)}_1 = \mathfrak{AR}^{\Lambda}_1 \setminus \{q_1,q_2\}$. 
        \item\label{lem:DK_seq} Any Auslander--Reiten sequence in $\leftmod{\Lambda/\mathrm{soc}(U)}$ is an Auslander--Reiten sequence in $\leftmod{\Lambda}$. Up to isomorphism, exactly one Auslander--Reiten sequence in $\leftmod{\Lambda}$ is not an Auslander--Reiten sequence in $\leftmod{\Lambda/\mathrm{soc}(U)}$. It is of the form 
        \begin{equation}\label{eq:DKseq}
            \begin{tikzcd}[column sep=4em,ampersand replacement=\&]
                0 \arrow[r] \& \mathrm{rad}(U) \arrow[r,"{\begin{pmatrix}
                    q_1 \\ f
                \end{pmatrix}}"] \& U\oplus X \arrow[r,"{\begin{pmatrix}
                    q_2 & -g
                \end{pmatrix}}"] \& U/\mathrm{soc}(U)\arrow[r] \& 0,
            \end{tikzcd}
        \end{equation}
        where $X$ need not be indecomposable.
        \item\label{lem:DK_newproj_inj} The $\Lambda$-module $U/\mathrm{soc}(U)$ is an indecomposable projective $\Lambda/\mathrm{soc}(U)$-module.
    \end{enumerate}
\end{lemma}

\begin{example}\label{ex:A3cont}
Let $Q$ be as in \Cref{ex:A3}, and consider the indecomposable projective-injective $kQ$-module $U=\begin{smallmatrix} 1 \\ 2 \\ 3 \end{smallmatrix}$. The ideal $\mathrm{soc}(U)$ is generated by the path $ba$. Hence, the Auslander--Reiten quiver of $\leftmod{kQ/\mathrm{soc}(U)}$ is 
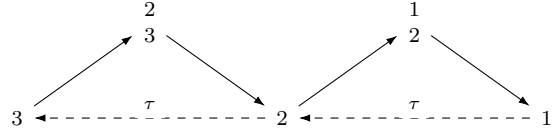
   \[ \begin{tikzpicture}[x=0.035cm,y=-0.025cm]
  \clip (0,410) rectangle + (240,75);
  \begin{scope}[every node/.style={inner sep=4pt,font=\footnotesize}]
   \node (2_0) at (70,427) {$\begin{matrix}  2 \\ 3 \end{matrix}$};
   \node (3_0) at (20,477) {$3$};
   \node (3_1) at (120,477) {$2$};
   \node (2_1) at (170,427) {$\begin{matrix}  1 \\ 2 \end{matrix}$};
   \node (3_2) at (220,477) {$1$};
  \end{scope}
  \begin{scope}[every node/.style={fill=white,font=\scriptsize},every path/.style={-{Latex[length=1.5mm,width=1mm]}}]
   \path (3_0) edge (2_0);
   \path (3_1) edge[dashed, "$\tau$"'] (3_0);
   \path (2_0) edge (3_1);
   \path (3_1) edge (2_1);
   \path (3_2) edge[dashed, "$\tau$"'] (3_1);
   \path (2_1) edge (3_2);
  \end{scope}
 \end{tikzpicture}\] 
 \end{example}

The following statement is a consequence of \Cref{lem:DK} and \Cref{lem:MagARQuiver}.
\begin{lemma}\label{lem:DK_magnitude}
    Let $\Lambda=kQ/I$ be a bound path $k$-algebra of rank $n$ and let $U$ be a basic projective-injective $\Lambda$-module without simple direct summands. Then $\MagnitudeOfModuleCat{\Lambda} = \MagnitudeOfModuleCat{\Lambda / \mathrm{soc}(U)}$. Moreover, the rank of $\Lambda/\mathrm{soc}(U)$ is $n$.
\end{lemma}
\begin{proof}
	We proceed by induction on the number $t$ of pairwise non-isomorphic indecomposable direct summands of $U$. The case $t=0$ is trivial, so suppose that $t=1$. By \Cref{lem:DK}, we obtain the Auslander--Reiten quiver of $\leftmod{\Lambda/ \mathrm{soc}(U)}$ by removing $U$, its two adjacent arrows, and the Auslander--Reiten sequence displayed in \eqref{eq:DKseq}. In particular, we lose the projective $\Lambda$-module $U$, but turn $U / \mathrm{soc}(U)$ projective. This shows that the rank remains unchanged. {Using} \Cref{lem:MagARQuiver}, {this discussion also yields} that $\MagnitudeOfModuleCat{\Lambda}  = \MagnitudeOfModuleCat{\Lambda/\mathrm{soc}(U)}$. For the inductive step, suppose that $t\geq 2$. We may therefore decompose $U$ into $U' \oplus U_1$, where $U_1$ is indecomposable and not isomorphic to a direct summand of $U'$. By \Cref{lem:DK}\eqref{lem:DK_seq}, it follows that $U'$ is a projective-injective $\Lambda/\mathrm{soc}(U_1)$-module. As a result of the Third Isomorphism Theorem, we have a $k$-algebra isomorphism
	\begin{equation*}
		{\Lambda/\mathrm{soc}(U_1) \over {\mathrm{soc}(U')}} \simeq {\Lambda/\mathrm{soc}(U_1) \over {\mathrm{soc}(U)/\mathrm{soc}(U_1)}} \simeq {\Lambda \over \mathrm{soc}(U)}.
	\end{equation*}
	Hence, we deduce that 
	\begin{equation*}
	\MagnitudeOfModuleCat{\Lambda}  = \MagnitudeOfModuleCat{\Lambda/\mathrm{soc}(U_1)} =  \MagnitudeOfModuleCat{\Lambda/\mathrm{soc}(U)},
	\end{equation*}
	where the first equality comes about from the anchor step of the induction and the second follows from the induction hypothesis. Using a similar argument, one shows that the rank also remains unchanged.
\end{proof}

When calculating the magnitude of module categories, it is often useful to apply \Cref{lem:DK_magnitude} iteratively. Continuing with our \Cref{ex:A3cont}, one may factor out the socle of the maximal basic projective-injective $kQ/\mathrm{soc}(U)$-module, namely $\begin{smallmatrix}  2 \\ 3 \end{smallmatrix} \oplus \begin{smallmatrix}  1 \\ 2 \end{smallmatrix}$. As a result, we obtain a semisimple $k$-algebra of rank $3$, which has magnitude equal to $3$. By \Cref{lem:DK_magnitude}, this is an alternative way of showing that $\MagnitudeOfModuleCat{kQ}=3$.

\section{Euler characteristic of translation quivers}\label{sec: Translation quivers}

In the previous section, we showed that computing the magnitude of Auslander algebras reduces to a combinatorial computation on the Auslander--Reiten quiver. We will spend this section developing a theory of Euler characteristic for more general translation quivers. This will be useful when working with hereditary and self-injective algebras.

\begin{definition}[{see \cite{Rie80}, \cite[§VII.4]{ARS97}}]\label{def:translationQuiver}
 Let $\mathfrak{T}=(\mathfrak{T}_0,\mathfrak{T}_1)$ be a locally finite quiver without loops and multiple arrows (i.e.\, for any $x,y\in \mathfrak{T}_0$, there is at most one arrow $x \rightarrow y$ in $\mathfrak{T}_1$).
\begin{enumerate}
	\item Let $x\in \mathfrak{T}_0$. We define the following subsets of $\mathfrak{T}_0$:
	\begin{align*}
		x^{-} &\coloneqq \{y \in \mathfrak{T}_0 \,:\, \exists \text{ an arrow } y \rightarrow x \text{ in $\mathfrak{T}_1$} \},  \\
		x^{+} &\coloneqq  \{y \in \mathfrak{T}_0 \,:\, \exists \text{ an arrow } x \rightarrow y \text{ in $\mathfrak{T}_1$} \}.
	\end{align*}
	The elements in $x^{-}$ are the \textit{immediate predecessors of $x$} and elements in $x^{+}$ are the \textit{immediate successors of $x$}.
	\item A \textit{translation} of $\mathfrak{T}$
    is an injection $\tau\colon \mathfrak{T}^{\mathrm{np}}_0 \rightarrow \mathfrak{T}_0$, where $\mathfrak{T}^{\mathrm{np}}_0$ is a subset of $\mathfrak{T}_0$, such that $\tau(x)^+ = x^-$ for all $x\in \mathfrak{T}^{\mathrm{np}}_0 $ (see \eqref{eq:mesh} below). We refer to the pair $(\mathfrak{T},\tau)$ as a \textit{(non-valued) translation quiver} (without multiple arrows, see \cite[IV.4]{ASS06}). The vertices in $\mathfrak{T}^{\mathrm{np}}_0$ are called \textit{non-projective}, whereas the remaining vertices are called \textit{projective}.
    \end{enumerate}
    Let $(\mathfrak{T},\tau)$ be a translation quiver.
    \begin{enumerate}
    \setcounter{enumi}{2}
	\item We say that $(\mathfrak{T},\tau)$ is \textit{proper} if $x^-$ is non-empty for every $x\in \mathfrak{T}_0^{\mathrm{np}}$.
	\item\label{def:meshes_alpha} For $x\in \mathfrak{T}_0^{\mathrm{np}}$, denote the cardinality of $x^-$ by $\alpha(x)$. We also define
	$
		\alpha(\mathfrak{T},\tau) \coloneqq \sup_{x\in \mathfrak{T}_0^{\mathrm{np}}} \alpha(x).
	$
	\item\label{def:meshes_meshes} A \textit{mesh} in $(\mathfrak{T},\tau)$ is given by a {full} subquiver of $(\mathfrak{T},\tau)$ of the form \begin{equation}\label{eq:mesh}
    \begin{tikzpicture}[x=0.035cm,y=-0.020cm]
  \clip (0,365) rectangle + (200,120);
  \begin{scope}[every node/.style={inner sep=4pt}]
   \node (1_1) at (120,377) {$y_1$};
   \node (2_1) at (70,427) {$\tau(x)$};
   \node (3_2) at (120,477) {$y_{\alpha(x)}$};
   \node (2_2) at (170,427) {$x$};
   \node (2_1.5) at (120,427) {$\vdots$};
  \end{scope}
  \begin{scope}[every node/.style={fill=white,font=\scriptsize},every path/.style={-{Latex[length=1.5mm,width=1mm]}}]
   \path (2_1) edge["$g_{\alpha(x)}$"'] (3_2);
   \path (2_1) edge["$g_1$"] (1_1);
   \path (3_2) edge["$f_{\alpha(x)}$"'] (2_2);
   \path (1_1) edge["$f_{1}$"] (2_2);
  \end{scope}
 \end{tikzpicture} 
	\end{equation}
    where {$x\in \mathfrak{T}_0^{\mathrm{np}}$} and $x^- = \{y_1,\dots,y_{\alpha(x)}\}$. 
    \item\label{def:translationQuiverMor} A \textit{morphism of translation quivers} from $(\mathfrak{T},\tau)$ to $(\mathfrak{T}',\tau')$ is given by a morphism of quivers $\mathfrak{T}\xrightarrow{\varphi} \mathfrak{T}'$ such that $\varphi(\tau x) = \tau' \varphi(x)$ for all $x\in \mathfrak{T}_0^{\mathrm{np}}$. It is an \textit{isomorphism} if it admits an inverse morphism of translation quivers, and an \textit{automorphism} if additionally $(\mathfrak{T}',\tau')=(\mathfrak{T},\tau)$.
	\end{enumerate}
\end{definition}

Examples of (proper) translation quivers are Auslander--Reiten quivers of module categories of representation-finite bound path $k$-algebras. Projective vertices correspond to isomorphism classes of indecomposable projective modules, and meshes are determined by Auslander--Reiten sequences. 

\begin{definition}
	Let $(\mathfrak{T},\tau)$ be a finite translation quiver, and let $\mathcal{E}(\mathfrak{T},\tau)$ denote the set of meshes in $(\mathfrak{T},\tau)$. We then define the \textit{Auslander--Reiten--Euler characteristic} of $(\mathfrak{T},\tau)$ as follows:
	\begin{equation*}
		\chi_{\mathrm{AR}}(\mathfrak{T},\tau) \coloneqq |\mathfrak{T}_0| - |\mathfrak{T}_1| + |\mathcal{E}(\mathfrak{T},\tau)|.
	\end{equation*}
\end{definition}

\begin{remark}\label{rem:chi=chiAR}
    Let $(\mathfrak{AR}^{\Lambda},\tau)$ be the Auslander--Reiten quiver of $\leftmod{\Lambda}$ for a representation-finite bound path $k$-algebra $\Lambda$. It follows from \Cref{lem:MagARQuiver} that $\MagnitudeOfModuleCat{\Lambda}=\chi_{\mathrm{AR}}(\mathfrak{AR}^{\Lambda},\tau)$.
    \end{remark}
    
\begin{proposition}\label{prop:UpperBoundT}
\, 
    \begin{enumerate}
        \item\label{prop:UpperBoundT1} Let $(\mathfrak{T},\tau)$ be a finite proper translation quiver with $N$ vertices and $\ell$ arrows into projective vertices. Then $\chi_\mathrm{AR}(\mathfrak{T},\tau)\leq N-\ell$.
        \item\label{prop:UpperBoundT2} Let $\Lambda=kQ/I$ be a representation-finite bound path $k$-algebra over which there are $N$ indecomposable modules up to isomorphism. Then $\MagnitudeOfModuleCat{\Lambda}\leq N$, with equality if and only if $\Lambda$ is semisimple.
    \end{enumerate}
\end{proposition}
\begin{proof}
    We first prove \eqref{prop:UpperBoundT1}.
    Let $A$ denote the number of arrows and $E$ the number of meshes in $(\mathfrak{T},\tau)$.
    Observe that
    \begin{equation}\label{eq:UpperBoundEq}
        A = \ell + \sum_{x\in \mathfrak{T}_0^{\mathrm{np}}} \alpha(x),
    \end{equation}
    where the notation $\alpha(x)$ was introduced in \Cref{def:translationQuiver}.
    Since $(\mathfrak{T},\tau)$ is proper, we have $\alpha(x)\geq 1$ for all $x\in \mathfrak{T}_0^{\mathrm{np}}$. Also, since $|\mathfrak{T}_0^{\mathrm{np}}|=E$, it is now clear that 
    \begin{equation}\label{eq:AlE}
        A\geq \ell+E.
    \end{equation}
    Hence, we have 
    \[\chi_\mathrm{AR}(\mathfrak{T},\tau) = N - A + E \leq N-\ell,\]
    as claimed.

    Secondly, we prove \eqref{prop:UpperBoundT2}. 
    The Auslander--Reiten quiver of $\leftmod{\Lambda}$ is a finite proper translation quiver. Hence, we may keep the same notation as above.
    By \Cref{rem:chi=chiAR} and \eqref{prop:UpperBoundT1}, we have that $\MagnitudeOfModuleCat{\Lambda}\leq N$.    
    It now suffices to show that $A=E$ precisely when $\Lambda$ is semisimple.
    If $\Lambda$ is semisimple, then $A=0=E$. Conversely, {if $A=E$}, we apply \eqref{eq:AlE} to deduce that it is necessary for $\ell$ to be $0$. Since $\ell$ is equal to the number of indecomposable direct summands of $\mathrm{rad}(\Lambda)$ up to isomorphism (see \cite[Remark IV.4.3(a)]{ASS06}), it follows that $\Lambda$ has trivial radical, so $\Lambda$ is semisimple.
\end{proof}

As shown in \Cref{sec: Magnitude of Auslander algebras}, Auslander--Reiten quivers are the main gadget for computing the magnitude of module categories. There are nevertheless other examples of translation quivers that will be useful for our purpose. 

\begin{definition}\label{def:repquiver}
    Let $T=(T_0,T_1)$ be a locally finite quiver {without loops} and multiple arrows. The \textit{repetitive quiver} of $T$ is the translation quiver denoted $\Z T$, which is defined as follows:
    \begin{itemize}
        \item $(\Z T)_0 \coloneqq \Z \times T_0  = \{(d,x) \,:\, d\in \Z, \,  x\in T_0 \}$,
        \item $(\Z T)_1 \coloneqq (\Z \times T_1) \cup \{(d,y) \xrightarrow{(d,a)^\ast} (d+1,x) \,:\, x\xrightarrow{a} y \in T_1,\, d\in\Z \}$,
        \item $\tau \colon (\Z T)_0 \to (\Z T)_0 $ such that $\tau(d,x)=(d-1,x)$. 
    \end{itemize}
    For $t\geq 1$, define the \textit{$t$-bounded repetitive quiver of $T$}, denoted $tT$, to be the full sub-translation quiver of $\Z T$ containing the vertices in $(\Z\cap [0,t-1])\times T_0$. 
\end{definition}

It is easy to see that $\Z T$ is independent of the orientation of $T$ up to isomorphism of translation quivers \cite[Satz on p. 204]{Rie80}.

As an example, let $Q$ be the Dynkin quiver of type $A_3$ considered in \Cref{ex:A3}, and let $Q^{\mathrm{op}}$ denote its opposite quiver. The translation quiver $3Q^{\mathrm{op}}$ then {has the form} 
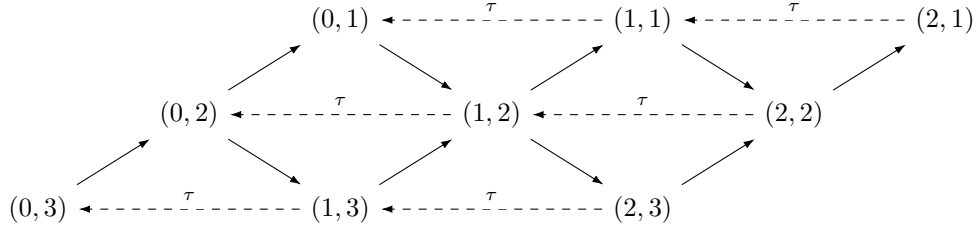
\[ \begin{tikzpicture}[x=0.040cm,y=-0.025cm]
  \clip (0,365) rectangle + (335,120);
  \begin{scope}[every node/.style={inner sep=4pt}]
   \node (1_1) at (120,377) {$(0,1)$};
   \node (2_1) at (70,427) {$(0,2)$};
   \node (3_1) at (20,477) {$(0,3)$};
   \node (3_2) at (120,477) {$(1,3)$};
   \node (2_2) at (170,427) {$(1,2)$};
   \node (3_3) at (220,477) {$(2,3)$};
   \node (1_2) at (220,377) {$(1,1)$};
   \node (2_3) at (270,427) {$(2,2)$};
   \node (1_3) at (320,377) {$(2,1)$};
  \end{scope}
  \begin{scope}[every node/.style={fill=white,font=\scriptsize},every path/.style={-{Latex[length=1.5mm,width=1mm]}}]
   \path (3_1) edge (2_1);
   \path (3_2) edge[dashed, "$\tau$"'] (3_1);
   \path (2_1) edge (3_2);
   \path (2_1) edge (1_1);
   \path (2_2) edge[dashed, "$\tau$"'] (2_1);
   \path (3_2) edge (2_2);
   \path (1_1) edge (2_2);
   \path (3_3) edge[dashed, "$\tau$"'] (3_2);
   \path (2_2) edge (3_3);
   \path (2_2) edge (1_2);
   \path (1_2) edge (2_3);
   \path (3_3) edge (2_3);
   \path (2_3) edge (1_3);
   \path (1_3) edge[dashed, "$\tau$"'] (1_2);
   \path (1_2) edge[dashed, "$\tau$"'] (1_1);
   \path (2_3) edge[dashed, "$\tau$"'] (2_2);
  \end{scope}
 \end{tikzpicture}\] 
The full sub-translation quiver spanned by the vertices $(i,j)$ with $i<j$ is isomorphic to the Auslander--Reiten quiver shown in \Cref{ex:A3}.

Finally in this section, we include a result on the Auslander--Reiten--Euler characteristic of bounded repetitive quivers. It will be useful in \Cref{sec:Dynkin,sec:SelfInj}.

\begin{lemma}\label{lem:Magnitude of repetitions}
    Let $T=(T_0,T_1)$ be a {finite} quiver {without loops} and multiple arrows, let $t\geq 1$, and let $tT$ be the $t$-bounded repetitive quiver of $T$. Then we have $\chi_{\mathrm{AR}}(tT)=(2t-1)(|T_0|-|T_1|)$. 
    In particular, if $T$ is a tree, then $\chi_{\mathrm{AR}}(tT)=2t-1$.
\end{lemma}
\begin{proof}
    The number of vertices in $tT$ is $t|T_0|$. Since there are $|T_1|$ arrows between every two consecutive copies of $T$ in $tT$, there are
    $
        t |T_1| + (t-1) |T_1| = (2t-1) |T_1|
    $
    arrows in $tT$ in total. Since the number of meshes in $tT$ is $(t-1)|T_0|$, we conclude that
    \[
        \chi_{\mathrm{AR}}(tT) = t|T_0| - (2t-1) |T_1| + (t-1)|T_0| = (2t-1)(|T_0|-|T_1|).
   \qedhere \]
\end{proof}

\section{Biserial algebras}\label{sec:biserial}

\begin{definition}\label{def: special biserial and string}
    A finite-dimensional bound path $k$-algebra $\Lambda=kQ/I$ is \textit{special biserial} if the following conditions hold. 
    \begin{enumerate}
        \item Every vertex in $Q$ has in-degree at most 2 and out-degree at most 2. 
        \item  For every arrow $i \xrightarrow{a} j$ in $Q$, there is at most one arrow $j \xrightarrow{b} j'$ in $Q$ such that $ba \not\in I$, and there is at most one arrow $i' \xrightarrow{c} i$ such that $ac\not\in I$. 
    \end{enumerate}
     If, additionally, the ideal $I$ is generated by paths, we say that $\Lambda$ is a \textit{string algebra}. 
\end{definition}

\begin{remark}
Recall that a module is \textit{uniserial} if it has a unique composition series.
Special biserial algebras are \textit{biserial}, in the sense of Fuller \cite{Ful79}. This is to say that all indecomposable projective left and right modules $P$ satisfy both of the following conditions:
\begin{enumerate}
    \item There exist uniserial submodules $M_1$ and $M_2$ of $P$ which either intersect trivially or at a simple submodule of $P$.
    \item There exists a unique maximal submodule $N$ of $P$, and $M_1 + M_2=N$.
\end{enumerate}
A representation-finite bound path $k$-algebra is biserial if and only if it is special biserial \cite{SW83}. 
\end{remark}

We recite well-known characterisations of representation-finite (special) biserial algebras and string algebras in the next lemma.
 
\begin{proposition}
    [{\cite[Corollary of Theorem 1]{SW83}, see also \cite[Theorem on p. 175]{BR87} and \cite[Proposition 2]{SA23}}]\label{lem:WW}
Fixing a representation-finite $k$-algebra $\Lambda$, let $\alpha(\Lambda)$ denote the maximal number of middle terms in an Auslander--Reiten sequence in $\leftmod{\Lambda}$ (cf. \Cref{def:translationQuiver}\eqref{def:meshes_alpha}). Furthermore, let $\beta(\Lambda)$ denote the maximal number of non-projective middle terms of such Auslander--Reiten sequences.
\begin{enumerate}
\item\label{lem:ARseq} $\Lambda$ is special biserial if and only if $\beta(\Lambda)\leq 2$.
\item\label{lem:WWseq} $\Lambda$ is a string algebra if and only if $\alpha(\Lambda)\leq 2$. 
\item\label{lem:WWuniserial}  If $\Lambda$ is special biserial and $U$ denotes the direct sum of all indecomposable non-uniserial projective-injective $\Lambda$-modules up to isomorphism, then $\Lambda/\mathrm{soc}(U)$ is a string algebra.
\end{enumerate}
\end{proposition}

To prove \Cref{thm:SpecialBiserial}, we will need to classify the Auslander--Reiten sequences of modules over string algebras having exactly one middle term. Such a classification is also well established.

\begin{proposition}[{\cite[Theorem on p.\! 148, Lemma on p.\! 151, and Corollary on p.\! 174]{BR87}}]\label{Butler--Ringel}
	Let $\Lambda=kQ/I$ be a bound path $k$-algebra. For an arrow $i \xrightarrow{a} j$ in $Q$, let $V_{a}$ denote the indecomposable $\Lambda$-module with projective presentation $P_j \xrightarrow{a\cdot -} P_i$. Let $\mathrm{BR}_{\Lambda}(a)$ denote the isomorphism class of Auslander--Reiten sequences in $\leftmod{\Lambda}$ in which the final term is isomorphic to $V_{a}$. Then $\alpha(V_{a})=1$, and this construction determines an injective map
	\begin{equation*}
    \begin{tikzcd}
        Q_1 \arrow[r,"\mathrm{BR}_{\Lambda}"] & \mathcal{E}_1(\Lambda).
    \end{tikzcd}
	\end{equation*}
	If $\Lambda$ is a representation-finite string algebra, then $\mathrm{BR}_{\Lambda}$ is a bijection.
\end{proposition}

\begin{theorem}\label{thm:SpecialBiserial}
    Let $\Lambda$ be a representation-finite (special) biserial bound path $k$-algebra of rank $n$. Then $\MagnitudeOfModuleCat{\Lambda} = n$.
\end{theorem}
\begin{proof}
    Since simple modules are uniserial, we may apply \Cref{lem:DK_magnitude} and \Cref{lem:WW}\eqref{lem:WWuniserial} to $\Lambda$ and obtain a string algebra with a module category of equal magnitude. It is thus sufficient to consider the case where $\Lambda$ is a representation-finite string algebra of rank $n$. {Note that string algebras are monomial algebras by definition.} It then follows from \Cref{formula with 1-meshes and 3-meshes} and \Cref{lem:WW}\eqref{lem:WWseq} that
    $$\MagnitudeOfModuleCat{\Lambda} = n - |Q_1| + |\mathcal{E}_1(\Lambda)| .$$
    By \Cref{Butler--Ringel}, we have $|Q_1| = |\mathcal{E}_1(\Lambda)|  $, and thus, $\MagnitudeOfModuleCat{\Lambda} = n$.
\end{proof}

Recall that a finite-dimensional $k$-algebra $\Lambda$ is of \textit{left local-colocal type} if every indecomposable left $\Lambda$-module either has a simple top or a simple socle. It is stronger for $\Lambda$ to be of \textit{left local type} (resp.\! \textit{left colocal type}), i.e.\! every indecomposable left $\Lambda$-module has a simple top (resp.\! simple socle). Stronger still, we say that $\Lambda$ is a \textit{Nakayama $k$-algebra} if each indecomposable left $\Lambda$-module is uniserial.

\begin{corollary}\label{cor:Nakayama}
    Let $\Lambda$ be a representation-finite bound path $k$-algebra of rank $n$ which {is} of left local-colocal type. Then $\MagnitudeOfModuleCat{\Lambda} = n$. In particular, this holds when $\Lambda$ is a Nakayama $k$-algebra.
\end{corollary}
\begin{proof}
    Since $k$-algebras of left local-colocal type are biserial \cite[Proposition 2.7]{Tac61}, the assertion follows from \Cref{thm:SpecialBiserial}.
\end{proof}

\begin{remark}
    A less general version of \Cref{cor:Nakayama}, concerning representation-finite $k$-algebras that are of left local type (or dually, of left colocal type), can be proved from first principles as follows. Note that this argument also applies to Nakayama $k$-algebras.

    It follows from \Cref{rem:chi_indmod}\eqref{rem:chi_indmod1} that the similarity matrix of $\mathrm{Aus}(\Lambda)$ is $\mathbf{Z}_{\mathrm{Aus}(\Lambda)} = (\dim_k\Hom_{\Lambda}(M_j,M_i))_{i,j}$, where $\{M_j\}_{j=1}^N$ is a complete set of indecomposable $\Lambda$-modules up to isomorphism. Suppose that the first $n$ elements of this family are the simple $\Lambda$-modules $S_i$.
    Consider the $1\times N$ row vector $\mathbf{s}$, with 1 in the first $n$ entries and 0 elsewhere. By \Cref{rem:Cartan}\eqref{rem:Cartan2}, it is enough to show that $\mathbf{s}$ is a co-weighting of $\mathbf{Z}_{\mathrm{Aus}(\Lambda)}$. The first $n$ rows $r_i$ of $\mathbf{Z}_{\mathrm{Aus}(\Lambda)}$ are of the form $r_i=(\dim_k\Hom_{\Lambda}(M_j,S_i))_{j=1}^N$.
    Since all indecomposable $\Lambda$-modules have a simple top, it follows that each column of $\mathbf{Z}_{\mathrm{Aus}(\Lambda)}$ has a single non-zero entry among the first $n$ entries, and that this non-zero entry is 1. This implies that $\mathbf{s}$ is indeed a co-weighting of $\mathbf{Z}_{\mathrm{Aus}(\Lambda)}$.
\end{remark}

\begin{corollary}\label{cor:Local}
    Let $\Lambda$ be a representation-finite bound path $k$-algebra which is local, i.e.\! it has rank 1. Then $\MagnitudeOfModuleCat{\Lambda} = 1$.
\end{corollary}
\begin{proof}
    A local representation-finite bound path $k$-algebra is Morita equivalent to a $k$-algebra of the form $k[\begin{tikzcd} 1 \arrow[loop right,"a"] \end{tikzcd}]/(a^{\ell})$ for some $\ell\geq 1$. Since this is a Nakayama $k$-algebra, the claim follows from \Cref{cor:Nakayama}.
\end{proof}

It is also possible to prove \Cref{cor:Local} by iteratively using \Cref{lem:DK_magnitude}.
We conclude this section with classes of examples from various areas of representation theory.

\begin{corollary}\label{cor:block}
    Suppose that $k$ is algebraically closed and of prime characteristic $p$. 
    Let $G$ be a finite group whose order is divisible by $p$, and let $B$ be a $p$-block of the group algebra $kG$ with cyclic defect
group of order $p^d$, for some $d$. Then $\MagnitudeOfModuleCat{B}$ is equal to the number of irreducible Brauer characters of $B$. 
\end{corollary}
\begin{proof}
    The $k$-algebra $B$ is a Brauer tree $k$-algebra (see \cite[Theorem 17.1]{Alp86}). Since the class of Brauer tree $k$-algebras is known to coincide with the class of representation-finite symmetric biserial $k$-algebras \cite{GR79,SW83}, it follows from \Cref{thm:SpecialBiserial} that $\MagnitudeOfModuleCat{B}$ is equal to the rank of $B$. To conclude, we use that there are as many indecomposable projective $B$-modules as there are irreducible Brauer characters of $B$  \cite{Dad66,Jan69}.
\end{proof}

Representation-finite special biserial algebras appear frequently in representation theory. Various examples are listed in our next result (cf. \cite[§6.1]{CM19}).

\begin{corollary}
Let $n\geq 2$ and let $B$ be taken from any of the classes of $k$-algebras listed below. Then $\MagnitudeOfModuleCat{B}=n$. 
    \begin{enumerate}
        \item When $k$ is algebraically closed, a representation-finite block of a Schur $k$-algebra admitting $n$ simple modules up to isomorphism.
        \item\label{cor:gendoWes} When $k$ is of characteristic $0$, a block of a Temperley--Lieb $k$-algebra admitting $n$ simple modules up to isomorphism.
        \item\label{cor:gendoMar}  When $k$ is of characteristic $0$, a block of a partition $k$-algebra admitting $n$ simple modules up to isomorphism.
        \item For $k=\mathbb{C}$, the principal block of the parabolic BGG category associated to the pair $(\mathfrak{gl}_n(\mathbb{C}),\mathfrak{p})$, where $\mathfrak{p}$ is the parabolic subalgebra of $\mathfrak{gl}_n(\mathbb{C})$ given by the sum of its Borel subalgebra and Levi subalgebra. 
    \end{enumerate}
\end{corollary}
\begin{proof}
    In all cases except for semisimple instances of $B$ in \eqref{cor:gendoWes} and \eqref{cor:gendoMar}, the $k$-algebra $B$ is isomorphic to the bound path $k$-algebra with Gabriel quiver 
    \begin{equation*}
     \begin{tikzpicture}[x=0.045cm,y=-0.025cm]
  \clip (0,235) rectangle + (240,55);
  \begin{scope}[every node/.style={inner sep=4pt}]
   \node (1) at (20,261) {$1$};
   \node (2) at (70,261) {$2$};
   \node (d) at (120,261) {$\cdots$};
   \node (n1) at (170,261) {$n-1$};
   \node (n) at (220,261) {$n$};
  \end{scope}
  \begin{scope}[every node/.style={fill=white,font=\footnotesize},every path/.style={-{Latex[length=1.5mm,width=1mm]}}]
    \path (1) edge["$a_1$",bend left=10] (2);
    \path (2) edge["$a_2$",bend left=10] (d);
    \path (d) edge["$a_{n-2}$",bend left=10] (n1);
    \path (n1) edge["$a_{n-1}$",bend left=10,pos=0.4] (n);
    \path (2) edge["$b_{1}$",bend left=10] (1);
     \path (d) edge["$b_{2}$",bend left=10] (2);
     \path (n1) edge["$b_{n-2}$",bend left=10] (d);
    \path (n) edge["$b_{n-1}$",bend left=10,pos=0.55] (n1);
  \end{scope}
 \end{tikzpicture}
 \end{equation*}
    bound by the relations $a_{n-1}b_{n-1}$, $a_{i-1}b_{i-1} - b_ia_i$, $a_{i}a_{i-1}$, and $b_{i-1}b_i$, where $2\leq i\leq n$ \cite{Kue17,Wes95,Mar96,BS11}. Since $B$ is representation-finite and special biserial in all cases, the results follow from \Cref{thm:SpecialBiserial}.
\end{proof}

\section{Hereditary path algebras and radical square zero algebras}\label{sec:Dynkin}

In this section, we will give a formula for the magnitude of the module category of a representation-finite path algebra over $k$. These are exactly the path $k$-algebras of disjoint unions of \textit{Dynkin quivers} \cite{Gab72}, i.e.\! quivers whose underlying unoriented graphs are simply-laced Dynkin diagrams. They are of the following form:
\begin{center}
    \begin{equation}\tag{$A_n$, $n\geq 1$}
        \begin{tikzcd}
        1 \arrow[r, no head] & 2 \arrow[r, no head] & \cdots \arrow[r, no head] & n
    \end{tikzcd}
    \end{equation}

    \begin{equation}\tag{$D_n$, $n\geq 4$}
        \begin{tikzcd}[row sep=0.5em]
        1 \arrow[rd, no head] &&&&  \\
        & 3 \arrow[r, no head] & \cdots \arrow[r, no head] & n-1 \arrow[r, no head] & n \\
        2 \arrow[ru, no head] &&&&
    \end{tikzcd}
    \end{equation}

    \begin{equation}\tag{$E_6$}
        \begin{tikzcd}
        1 \arrow[r, no head] & 2 \arrow[r, no head] & 3 \arrow[r, no head] \arrow[d, no head] & 4 \arrow[r, no head] & 5 \\
        & & 6 & & 
    \end{tikzcd}
    \end{equation}

    \begin{equation}\tag{$E_7$}
        \begin{tikzcd}
        1 \arrow[r, no head] & 2 \arrow[r, no head] & 3 \arrow[r, no head] \arrow[d, no head] & 4 \arrow[r, no head] & 5 \arrow[r, no head] & 6 \\
        & & 7 & & & 
    \end{tikzcd}
    \end{equation}

    \begin{equation}\tag{$E_8$}
        \begin{tikzcd}
        1 \arrow[r, no head] & 2 \arrow[r, no head] & 3 \arrow[r, no head] \arrow[d, no head] & 4 \arrow[r, no head] & 5 \arrow[r, no head] & 6 \arrow[r, no head] & 7 \\
        & & 8 & & & & 
    \end{tikzcd}
    \end{equation}
\end{center}
By \Cref{rem:chi_indmod}\eqref{Splitting formula}, it suffices to compute $\MagnitudeOfModuleCat{kQ}$, where $Q$ is a Dynkin quiver. In the case where $k$ is algebraically closed, all connected representation-finite hereditary $k$-algebras are Morita equivalent to bound $k$-path algebras of Dynkin quivers. We first show that $\MagnitudeOfModuleCat{kQ}$ is independent of the orientation of $Q$. 

\begin{lemma}\label{lem:APR}
    Let $Q$ and $Q'$ be Dynkin quivers of the same Dynkin type. Then $\MagnitudeOfModuleCat{kQ} = \MagnitudeOfModuleCat{kQ'}$.
\end{lemma}
\begin{proof}
    Let $\mathcal{C}_{Q}$ denote the cluster category of $Q$ \cite{BMRRT06}. It is defined as the orbit category $\mathcal{D}^b(\leftmod{kQ})/{\tau^{-1}\Sigma}$, where $\mathcal{D}^b(\leftmod{kQ})$ is the bounded derived category of $\leftmod{kQ}$, $\Sigma$ is the suspension endofunctor on $\mathcal{D}^b(\leftmod{kQ})$, and $\tau$ is the Auslander--Reiten translation on $\mathcal{D}^b(\leftmod{kQ})$. The composite functor 
    \begin{equation*}
    \begin{tikzcd}
       \leftmod{kQ} \arrow[r,hook] &  \mathcal{D}^b(\leftmod{kQ}) \arrow[r] & \mathcal{C}_{Q}  
    \end{tikzcd}
    \end{equation*}
    induces a morphism of translation quivers $(\mathfrak{AR}^{kQ},\tau) \xrightarrow{} (\mathfrak{C}_{Q},\tau)$. Here, $(\mathfrak{C}_{Q},\tau)$ denotes the Auslander--Reiten quiver of $\mathcal{C}_{Q}$. The image of this composite functor identifies $\mathfrak{AR}^{kQ}$ with a full subquiver of $\mathfrak{C}_{Q}$. Also, the full subquiver $\mathfrak{P}$ of $\mathfrak{C}_{Q}$ spanned by the vertices outside this image is isomorphic to $Q$. Thus, there are $|Q_0|$ vertices and $|Q_1|$ arrows in $\mathfrak{P}$. Moreover, there are exactly $|Q_1|$ arrows from  $\mathfrak{AR}^{kQ}$ to $\mathfrak{P}$, and $|Q_1|$ arrows from $\mathfrak{P}$ to $\mathfrak{AR}^{kQ}$. Since there are no projective vertices in $(\mathfrak{C}_{Q},\tau)$, there are $2|Q_0|$ meshes in $(\mathfrak{C}_Q,\tau)$ that are not in $(\mathfrak{AR}^{\Lambda},\tau)$. In combination with \Cref{rem:chi=chiAR}, it follows that 
    \begin{align*}
        \MagnitudeOfModuleCat{kQ} = \chi_{\mathrm{AR}}(\mathfrak{AR}^{kQ},\tau) & =\chi_{\mathrm{AR}}(\mathfrak{C}_{Q},\tau) - |Q_0| + 3|Q_1| - 2|Q_0| \\  &= \chi_{\mathrm{AR}}(\mathfrak{C}_{Q},\tau) - 3(|Q_0|-|Q_1|) \\ &= \chi_{\mathrm{AR}}(\mathfrak{C}_{Q},\tau) - 3.
    \end{align*}
    An identical argument shows that $\MagnitudeOfModuleCat{kQ'} = \chi_{\mathrm{AR}}(\mathfrak{C}_{Q'},\tau) - 3$. Since the cluster categories $\mathcal{C}_{Q}$ and $\mathcal{C}_{Q'}$ are equivalent as triangulated categories (see \cite[pp. 576--577]{BMRRT06}), the translation quivers $(\mathfrak{C}_Q,\tau)$ and $(\mathfrak{C}_{Q'},\tau)$ are isomorphic. It follows that $\chi_{\mathrm{AR}}(\mathfrak{C}_{Q},\tau)=\chi_{\mathrm{AR}}(\mathfrak{C}_{Q'},\tau)$. We conclude that $\MagnitudeOfModuleCat{kQ} = \MagnitudeOfModuleCat{kQ'}$.
\end{proof}

Recall that the \textit{mesh $k$-category} of a translation quiver $(\mathfrak{T},\tau)$ is the free $k$-category of $\mathfrak{T}$ modulo the ideal generated by the morphisms of the form $\sum_{i} f_ig_i$ determined by each mesh, as displayed in \eqref{eq:mesh}.

\begin{propositiondefinition}\label{propdef:Coxeter}
    Let $Q$ be a Dynkin quiver of type $\Delta$ with $n$ vertices. The following positive integers are all equal. This integer $h_{Q}$ is the \textit{Coxeter number of $Q$}. Since it only depends on $\Delta$ (see \ref{propdef:Coxeter1}), we may also refer to it as the \textit{Coxeter number of $\Delta$}. 
    \begin{enumerate}[label=$h_{\arabic*}$]
    \setlength\itemsep{0.5em}
        \item\label{propdef:Coxeter1} $\coloneqq$ the order of a Coxeter element in an irreducible Coxeter group of type $\Delta$,
        \item\label{propdef:Coxeter2} $\coloneqq {2N \over n}$, where $N$ is the number of indecomposable $kQ$-modules up to isomorphism,
        \item\label{propdef:Coxeter3} $\coloneqq$ $\dim_k M + 1$, where $M$ is the indecomposable $kQ$-module of largest $k$-dimension,
        \item\label{propdef:Coxeter4} $\coloneqq$ $m_Q+1$, where $m_Q$ is the smallest positive integer such that any path of length $m_{Q}$ in the repetitive quiver $\Z Q$ is zero in the mesh $k$-category of $\Z Q$.
    \end{enumerate}
\end{propositiondefinition}
\begin{proof}
    Recall Gabriel's Theorem, namely that the dimension vectors of indecomposable $kQ$-modules are in one-to-one correspondence with the positive roots in the root system of type $\Delta$ \cite{Gab72}. The equalities \ref{propdef:Coxeter1}=\ref{propdef:Coxeter2}=\ref{propdef:Coxeter3} then follow from standard results on Coxeter numbers (see \cite[VI §1.11 Proposition 33]{Bourbaki4a6} \cite[Proposition 3.18 and Theorem 3.20]{Hum92}, \cite[Theorem 7.4(4)]{Hub25}). The fact that \ref{propdef:Coxeter4} is equal to the preceding quantities is also known \cite{BLR81}.
\end{proof}

\begin{theorem}\label{thm:Dynkin}
    Let $Q$ be a Dynkin quiver of type $\Delta$.  Then $\MagnitudeOfModuleCat{kQ} = h_{Q}-1$, where $h_{Q}$ is the Coxeter number of $Q$.
    Explicitly:
    \begin{center}
        \centering
        \begin{tabular}{|c|c|c|}
        \hline
            \text{Dynkin type $\Delta$} & \text{Coxeter number $h_{Q}$} & $\MagnitudeOfModuleCat{kQ}$ \\ \hline
            $A_n$ & $n+1$ & $n$ \\
            $D_n$ & $2n-2$  & $2n-3$  \\
            $E_6$ & $12$ & $11$  \\
            $E_7$ & $18$ & $17$  \\
            $E_8$ & $30$ & $29$  \\ \hline
        \end{tabular}
     \end{center}
\end{theorem}
\begin{proof}
    Suppose first that $Q$ is of type $A_n$ for some $n\geq 1$. Our claim then follows from \Cref{thm:SpecialBiserial}. By \Cref{lem:APR}, we know that the magnitude of the module category of $kQ$ is not affected by re-orientation of the arrows in $Q$. In the remaining cases, we may therefore re-orient $Q$ so that all arrows point away from the vertex of degree 3. The Auslander--Reiten quiver of $\leftmod{kQ}$ then becomes isomorphic to {the ${h_{Q}\over 2}$-bounded repetitive quiver} ${h_{Q}\over 2}Q^\mathrm{op}$ (see \cite{Gab72}). By \Cref{rem:chi=chiAR}, the quantity $\MagnitudeOfModuleCat{kQ}$ can thus be calculated as the Auslander--Reiten--Euler characteristic of ${h_{Q}\over 2}Q^\mathrm{op}$. It now follows from \Cref{lem:Magnitude of repetitions} that
    \[
    	 \MagnitudeOfModuleCat{kQ} = \chi_{\mathrm{AR}}\left({h_{Q}\over 2}Q^\mathrm{op}\right) = 2\left({h_{Q} \over 2}\right) -1 = h_{Q}-1. \qedhere
    \]
\end{proof}

From the equivalent definitions of Coxeter numbers in \Cref{propdef:Coxeter}, various interpretations of $\MagnitudeOfModuleCat{kQ}$ follow. These include the quanities $\dim_k M$ and $m_Q$, as defined above. We offer another interpretation in terms of a ``weighted number of $\tau$-orbits'' in the Auslander--Reiten quiver.

\begin{lemma}\label{remark:Coxeter5}
    Let $Q$ be a Dynkin quiver of type $\Delta$. Consider the quantity
    \begin{equation}\label{def:w}
        w(Q) \coloneqq \sum_{\mathcal{O}} \sup\limits_{M\in \mathcal{O},\, 1\leq i\leq n } m_i,
    \end{equation}
    where the sum is indexed over the $\tau$-orbits of $\leftmod{kQ}$ and $(m_i)_i$ is the dimension vector of $M$. Then $w(Q)=\MagnitudeOfModuleCat{kQ}$.
\end{lemma}
\begin{proof}
    Recall that a \textit{reflection} at a sink $i$ in $Q$ reorients $Q$ by reversing all arrows incident to $i$. Denote this reflection by $Q_i$. Any two orientations of $\Delta$ can be related by a finite sequence reflections \cite[Theorem 1.2(1)]{BGP73}. To show that $w(Q)$ does not depend on the orientation of $Q$, it therefore suffices to show that $w(Q)=w(Q_i)$.
    In that case, one can use Auslander--Platzeck--Reiten tilting \cite{APR79} to reconstruct $\leftmod{kQ_i}$ from $\leftmod{kQ}$ by moving a simple module in the Auslander--Reiten quiver from one $\tau$-orbit to another. Thus, the terms on the right hand side of \eqref{def:w} remain the same, so $w(Q)=w(Q_i)$. 
    Choosing an arbitrary orientation of $Q$, it can hence be checked in a case-by-case manner that there is a bijection of multisets
    \begin{equation}
    \begin{tikzcd}
    	\{\delta_i \, |\, i\in Q_0\}\setminus\{1\} \arrow[r] & \left\{\sup\limits_{M\in \mathcal{O},\, 1\leq i\leq n } m_i \,\,\Bigg\vert\, 
    \,\mathcal{O}\right\},
    \end{tikzcd}
    \end{equation}
    where $\delta>0$ generates the radical of the Tits form of the extended Dynkin diagram of $\Delta$. We have that $\delta_i=\dim_k e_iM$, where $M$ is the indecomposable $kQ$-module of largest $k$-dimension (see \cite[§6]{Kac90}). Our claim now follows from \Cref{propdef:Coxeter} (see \ref{propdef:Coxeter3}) and \Cref{thm:Dynkin}.
\end{proof}

\begin{remark}\label{rem:tilting}
    \Cref{lem:APR} can also be proved by showing that Auslander--Platzeck--Reiten tilting of Dynkin quivers preserves the magnitude of module categories.
    However, it is not true that tilting preserves this quantity in general. As a counter-example, consider $k\big[ \begin{tikzcd}[row sep=0.5em]
        1 \arrow[r,"a"] & 2\arrow[r,"b"] & 3\arrow[r,"c"] & 4
    \end{tikzcd}\big]/(cba)$. This is a tilted $k$-algebra of type $D_4$, and also a Nakayama $k$-algebra. By  \Cref{thm:Dynkin,cor:Nakayama}, the magnitude of module categories is not preserved under tilting in this instance. Hence, the magnitude of module categories is not preserved by derived Morita equivalence \cite{Ric89}. Moreover, representation-finiteness is not necessarily preserved by derived equivalences. Hence, magnitude may not be defined for all algebras in a derived equivalence class. For example, there are several families of Nakayama $k$-algebras that are derived equivalent to representation-infinite hereditary $k$-algebras \cite{HS10,Fos24}.
\end{remark}

Next, we calculate the magnitude of the module category of \textit{radical square zero bound path $k$-algebras}. Recall that these are of the form $kQ/I$, where $I$ is generated by all paths of length 2 in $Q$. Let the vertex set of $Q$ be $\{1,\dots,n\}$. Then the \textit{separating quiver} $Q_{\mathrm{sep}}$ of $\Lambda$ is defined as the quiver with vertices $\{1^-,\dots,n^-,1^+,\dots,n^+\}$ and an arrow $i^- \xrightarrow{a}j^+$ for every arrow $i \xrightarrow{a}j$ in $Q$. 
We have that $\Lambda$ is representation-finite if and only if $Q_{\mathrm{sep}}$ is a disjoint union of Dynkin quivers (see \cite[Theorem X.2.6]{ARS97}). 

\begin{corollary}\label{cor:radsq0}
    Let $\Lambda$ be a representation-finite radical square zero bound path $k$-algebra of rank $n$. Then
    \begin{equation}\label{eq:radsq0}
        \MagnitudeOfModuleCat{\Lambda} = n+ \sum_{C} (h_C-|C_0|-1),
    \end{equation}
    where the sum is indexed over all connected components $C=(C_0,C_1)$ of $Q_{\mathrm{sep}}$ and $h_C$ denotes the Coxeter number of $C$. In particular, we have $\MagnitudeOfModuleCat{\Lambda}\geq n$.
    Moreover, the following assertions are equivalent.
    \begin{enumerate}
        \item\label{cor:radsq0_1} $\Lambda$ is a string algebra.
        \item\label{cor:radsq0_2} $\MagnitudeOfModuleCat{\Lambda} = n$.
        \item\label{cor:radsq0_3} $Q_{\mathrm{sep}}$ is a disjoint union of type $A$ Dynkin quivers.
    \end{enumerate}
\end{corollary}
\begin{proof}
    Let $N$ denote the number of indecomposable $\Lambda$-modules up to isomorphism, and let $A$ denote the number of arrows in $\mathfrak{AR}^\Lambda$.
    It is well-known that $\mathfrak{AR}^{kQ_{\mathrm{sep}}}$ has $N+n$ vertices and $A$ arrows \cite[§X.2]{ARS97} \cite[Theorem 4.2]{Dro26}. By construction, the rank of $kQ_{\mathrm{sep}}$ is $2n$. 
    It follows from \Cref{lem:MagARQuiver} that
    \begin{align*}
        \MagnitudeOfModuleCat{kQ_{\mathrm{sep}}} - \MagnitudeOfModuleCat{\Lambda} &= (2(N+n)-A-2n) - (2N-A-n) = n. 
    \end{align*}
    This shows that $\MagnitudeOfModuleCat{\Lambda} = \MagnitudeOfModuleCat{kQ_{\mathrm{sep}}}-n$.
    Using \Cref{rem:chi_indmod}\eqref{Splitting formula} together with \Cref{thm:Dynkin} yields $\MagnitudeOfModuleCat{kQ_{\mathrm{sep}}}=\sum_C(h_C-1)$.
    Moreover, we observe that $\sum_{C} |C_0| = 2n$ since both quantities count the rank of $kQ_{\mathrm{sep}}$.
    Combining our observations shows \eqref{eq:radsq0}.

    We now prove that the assertions \eqref{cor:radsq0_1}--\eqref{cor:radsq0_3} are equivalent. Note that the representation-finiteness of $\Lambda$ is a necessary assumption.
    It was shown in \Cref{thm:SpecialBiserial} that \eqref{cor:radsq0_1} implies \eqref{cor:radsq0_2}. Also, it follows from \eqref{eq:radsq0} that \eqref{cor:radsq0_2} implies \eqref{cor:radsq0_3}. Finally, we show that \eqref{cor:radsq0_3} implies \eqref{cor:radsq0_1}. Note that 
    a radical square zero bound path $k$-algebra is a string algebra if and only if all vertices in the Gabriel quiver have in-degree and out-degree at most 2 (see \Cref{def: special biserial and string}). It follows that if $\Lambda$ is not a string algebra, then $Q_{\mathrm{sep}}$ necessarily contains a connected component which is a Dynkin quiver of type $D$ or $E$ (since $\Lambda$ was assumed to be representation-finite).
\end{proof}

\section{Self-injective algebras}\label{sec:SelfInj}

\begin{definition}
    A $k$-algebra $\Lambda$ is \textit{self-injective} if $\Lambda$ is injective as a left $\Lambda$-module. 
\end{definition}

Representation-finite self-injective bound path algebras have a well understood Auslander--Reiten theory. This will enable us to compute the magnitude of their module categories. To give a sufficient description of the Auslander--Reiten quiver of the module category of a self-injective algebra, we need to recall some terminology.

\begin{definition}[{\cite{Rie80}}]
Let $(\mathfrak{T},\tau)$ be a translation quiver.
    \begin{enumerate}
        \item A vertex $x$ in $\mathfrak{T}$ is \textit{stable} if $\tau^{\ell}x$ is non-projective for all $\ell \geq 0$.
        \item The translation quiver $(\mathfrak{T},\tau)$ is \textit{stable} if every vertex in $\mathfrak{T}$ is stable. Equivalently, every vertex in $\mathfrak{T}$ is non-projective.
        \item Given a translation quiver $(\mathfrak{T},\tau)$, we define its \textit{stable part} to be the full sub-translation quiver of $(\mathfrak{T},\tau)$ consisting of the stable vertices in $\mathfrak{T}$.
    \end{enumerate}
\end{definition}

For example, repetitive quivers (as defined in \Cref{def:repquiver}) are stable translation quivers.
Another important class of stable translation quivers are given by \textit{stable Auslander--Reiten quivers}. These are by definition the stable parts of Auslander--Reiten quivers. 

\begin{remark}\label{rem:sAR}
    The case of self-injective bound path algebras is of particular interest because the stable Auslander--Reiten quiver contains all vertices of the Auslander--Reiten quiver except the indecomposable projectives. Indeed, if $M$ is a non-projective indecomposable, then $\tau^{\ell}M$ cannot possibly be projective for any $\ell\geq 0$, because there would otherwise exist a non-injective indecomposable projective.
\end{remark}

Before proving \Cref{thm:Self-injective}, we follow Riedtmann's approach to characterising the stable Auslander--Reiten quivers of self-injective algebras \cite{Rie80} over algebraically closed fields. By \Cref{rem:fielddoesntmatter}\eqref{rem:fdnm_magnitude},  the Auslander--Reiten quiver of a bound path algebra is unaltered under extension of scalars. We may therefore consider self-injective bound path algebras over an arbitrary field.

\begin{definition}
        Let $G$ be a group acting on a translation quiver $(\mathfrak{T},\tau)$ by automorphisms (see \Cref{def:translationQuiver}\eqref{def:translationQuiverMor}), i.e.\! a subgroup of the automorphism group of $(\mathfrak{T},\tau)$.  We say that the action of $G$ is \textit{admissible} if for every $x$ in $\mathfrak{T}_0$, each orbit of $G$ meets $\{x\} \cup x^-$ in at most one vertex, and meets $\{x\} \cup x^+$ in at most one vertex. 
\end{definition}

The translation $\tau$ of a stable translation quiver $(\mathfrak{T},\tau)$ determines an automorphism of $(\mathfrak{T},\tau)$. By \Cref{def:translationQuiver}\eqref{def:translationQuiverMor}, the automorphism $\tau$ is central in the automorphism group of $(\mathfrak{T},\tau)$.

If a group $G$ acts admissibly on a translation quiver $(\mathfrak{T},\tau)$, then the quotient $\mathfrak{T}/G$ again admits the structure of a translation quiver. We abuse notation by writing $\tau$ for the induced translation on $\mathfrak{T}/G$. We call $(\mathfrak{T}/G,\tau)$ the \textit{orbit translation quiver} of $\mathfrak{T}$ with respect to $G$.
If $(\mathfrak{T},\tau)$ is stable, so is $(\mathfrak{T}/G,{\tau})$. In particular, an admissible quotient $\Z T/G$ of a repetitive quiver $\Z T$ is a stable translation quiver. We now recall how Riedtmann characterises the stable Auslander--Reiten quivers of connected representation-finite bound path $k$-algebras. 

\begin{lemma}[{\cite[Hauptsatz and Anhang 2]{Rie80}}]\label{thm:Rie80}
    Let $\Lambda$ be a representation-finite bound path $k$-algebra. 
    \begin{enumerate}
        \item\label{thm:Rie80Haupt} For every connected component $(\mathfrak{L},\tau)$ of the stable Auslander--Reiten quiver of $\leftmod{\Lambda}$, there exists a Dynkin quiver $T$ and a non-trivial group $G$ acting admissibly on $\Z T$ such that $(\mathfrak{L},\tau)$ is isomorphic to $\Z T/G$ as a translation quiver.
        \item\label{thm:Rie80Anh2} Let ${T}$ be a Dynkin quiver. Then any non-trivial group acting admissibly on $\Z {T}$ is an infinite cyclic group. All possible generators are of the form $\sigma \tau^{-r}$, where $r\geq 1$ and $\sigma$ is an automorphism of $\Z {T}$ which fixes a vertex $x\in\Z T$.
        \item\label{thm:Rie80sAR} In particular, if $\Lambda$ is self-injective, connected, and not semisimple, then the stable Auslander--Reiten quiver of $\leftmod{\Lambda}$ is connected (see \Cref{rem:sAR}) and isomorphic to $\Z T/\langle \sigma \tau^{-r}\rangle$. Here, $T$ is a Dynkin quiver, $r\geq 1$, and $\sigma$ is an automorphism of $\Z {T}$ with a fixed vertex.
    \end{enumerate}
\end{lemma}

Let $T$ be a directed tree. Recall that a \textit{section} of $\Z T$ is a full subquiver of $\Z T$ containing exactly one vertex in each $\tau$-orbit. It is clear that automorphisms of $\Z T$ send sections to sections, and that only the identity automorphism sends a given section to itself.
Let $\sigma$ be as in \Cref{thm:Rie80}\eqref{thm:Rie80Anh2}, and let $x$ be a fixed vertex of $\sigma$. Since $x$ is contained in at least one section $S$ of $\Z T$, $\sigma$ sends $S$ to a (possibly different) section containing $x$. There are only finitely many sections of $\Z T$ containing $x$. Hence, there must exist positive integers $s\neq t$ such that $\sigma^sS=\sigma^tS$. This shows that $\sigma$ must be of finite order.

To prove \Cref{thm:Self-injective}, we need a result on the Auslander--Reiten--Euler characteristic of certain finite stable translation quivers. An illustrative example will follow in \Cref{ex:Self-injective_eu}.

\begin{lemma}\label{lem:Self-injective_eu}
    {Let $T$ be a directed tree and let $\sigma$ be an automorphism on $\Z T$ of finite order. Then $\chi_{\mathrm{AR}}(\mathbb{Z}T/\langle \sigma\tau^{-r}\rangle)=2r$ for any $r\geq 1$.}
\end{lemma}
\begin{proof} 
    We first compute the value of $\chi_{\mathrm{AR}}(\Z T/\langle \tau^{-t}\rangle)$ for an arbitrary $t\geq 1$. We may reconstruct $\Z T/\langle \tau^{-t}\rangle$ from the $t$-bounded repetitive quiver $tT$ by adding $|T_1|$ arrows and $|T_0|$ meshes. Since $T$ is a tree, we have by \Cref{lem:Magnitude of repetitions} that
    \begin{equation*}
       \chi_{\mathrm{AR}}(\Z T/\langle \tau^{-t}\rangle) =\chi_{\mathrm{AR}}(tT)-|T_1|+|T_0| =  \chi_{\mathrm{AR}}(tT) + 1 = 2t.
    \end{equation*}

    {Let $s$ denote the order of $\sigma$.} 
    Since $\tau$ is central in the automorphism group of $\Z T$, we have
    $(\sigma \tau^{-r})^s = \tau^{-sr}$. The group $\langle \sigma\tau^{-r} \rangle$ therefore admits a well-defined admissible action on $\Z T/\langle \tau^{-sr}\rangle$. The orbit translation quiver $(\Z T/\langle \tau^{-sr}\rangle)/ \langle \sigma \tau^{-r} \rangle$ is isomorphic to $\Z T/\langle \sigma \tau^{-r}\rangle$. Observe that $s$ is equal to the smallest positive integer $i$ such that $(\sigma \tau^{-r})^i$ has fixed points in $\Z T/\langle \tau^{-sr}\rangle$. Otherwise, one can show that $\sigma^i$ sends some $(d,x)\in \Z T$ to $(d',x)$ for some $d'\neq d$, which is impossible, as $\sigma$ has finite order. Thus, each $\langle \sigma\tau^{-r} \rangle$-orbit of vertices in  $\Z T/\langle \tau^{-sr}\rangle$, each $\langle \sigma\tau^{-r} \rangle$-orbit of morphisms in  $\Z T/\langle \tau^{-sr}\rangle$, and each $\langle \sigma\tau^{-r} \rangle$-orbit of meshes in  $\Z T/\langle \tau^{-sr}\rangle$, contain exactly $s$ elements. It follows that
    \[
         \chi_{\mathrm{AR}}(\Z T/\langle \tau^{-sr}\rangle) = s\cdot\chi_{\mathrm{AR}}(\Z T/\langle \sigma\tau^{-r}\rangle).
    \]
    Using the previous paragraph, we conclude that
    \begin{equation*}
       \chi_{\mathrm{AR}}(\Z T/\langle \sigma\tau^{-r}\rangle) = {\chi_{\mathrm{AR}}(\Z T/\langle \tau^{-sr}\rangle)\over s} = {2sr\over s} =  2r. \qedhere
    \end{equation*}
\end{proof} 

\begin{example}\label{ex:Self-injective_eu}
Let $T$ be the quiver \begin{tikzcd}1 \arrow[r]& 2 \arrow[r] & 3 & \arrow[l]4\arrow[r] &  5. \end{tikzcd} The stable translation quiver $\Z T/\langle \tau^{-6}\rangle$ can be displayed as follows
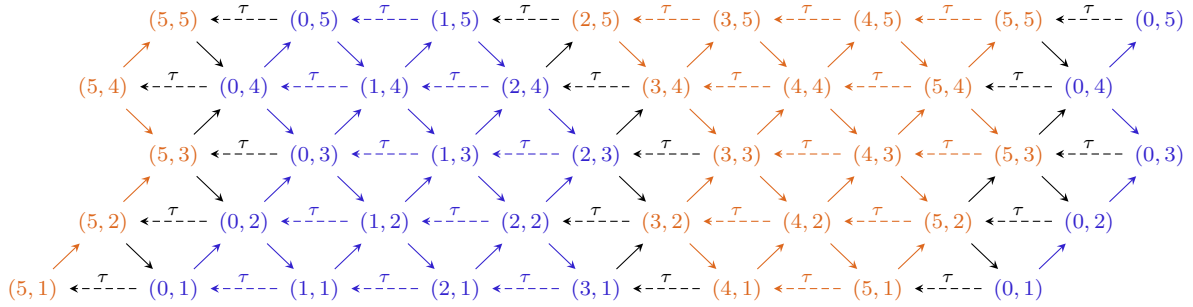
\[
\begin{tikzcd}[font=\footnotesize,row sep=1em, column sep=-0.1em]
                  &                              & \textcolor{friendly_beige}{(5,5)} \arrow[rd]                              &                                                           & {\textcolor{friendly_deepblue}{{(0,5)}}} \arrow[rd,color=friendly_deepblue] \arrow[ll, "\tau"', dashed]            &                                                           & \textcolor{friendly_deepblue}{(1,5)} \arrow[rd,color=friendly_deepblue] \arrow[ll, "\tau"', dashed,color=friendly_deepblue]            &                                                           & \textcolor{friendly_beige}{{(2,5)}} \arrow[ll, "\tau"', dashed] \arrow[rd,color=friendly_beige]            &                                                           & \textcolor{friendly_beige}{(3,5)} \arrow[rd,color=friendly_beige] \arrow[ll, "\tau"', dashed,color=friendly_beige]            &                                                           & \textcolor{friendly_beige}{(4,5)} \arrow[rd,color=friendly_beige] \arrow[ll, "\tau"', dashed,color=friendly_beige]            &                                                           & \textcolor{friendly_beige}{(5,5)} \arrow[rd] \arrow[ll, "\tau"', dashed,color=friendly_beige]            &                                                           & {\textcolor{friendly_deepblue}{{(0,5)}}} \arrow[ll, "\tau"', dashed] \\
                  &   \textcolor{friendly_beige}{(5,4)} \arrow[ru,color=friendly_beige] \arrow[rd,color=friendly_beige]                           &                                                & \textcolor{friendly_deepblue}{{(0,4)}} \arrow[ru,color=friendly_deepblue] \arrow[rd,color=friendly_deepblue] \arrow[ll, "\tau"', dashed]                           &                                                           & \textcolor{friendly_deepblue}{(1,4)} \arrow[ru,color=friendly_deepblue] \arrow[rd,color=friendly_deepblue] \arrow[ll, "\tau"', dashed,color=friendly_deepblue] &                                                           & \textcolor{friendly_deepblue}{(2,4)} \arrow[rd,color=friendly_deepblue] \arrow[ru] \arrow[ll, "\tau"', dashed,color=friendly_deepblue] &                                                           & \textcolor{friendly_beige}{{(3,4)}} \arrow[ru,color=friendly_beige] \arrow[rd,color=friendly_beige] \arrow[ll, "\tau"', dashed] &                                                           & \textcolor{friendly_beige}{(4,4)} \arrow[ru,color=friendly_beige] \arrow[rd,color=friendly_beige] \arrow[ll, "\tau"', dashed,color=friendly_beige] &                                                           & \textcolor{friendly_beige}{(5,4)} \arrow[ru,color=friendly_beige] \arrow[rd,color=friendly_beige] \arrow[ll, "\tau"', dashed,color=friendly_beige] &                                                           & {\textcolor{friendly_deepblue}{{(0,4)}}} \arrow[rd,color=friendly_deepblue] \arrow[ru,color=friendly_deepblue] \arrow[ll, "\tau"', dashed]  &                                    \\
                  &                              & \textcolor{friendly_beige}{(5,3)} \arrow[ru] \arrow[rd]                   &                                                           & \textcolor{friendly_deepblue}{{(0,3)}} \arrow[ru,color=friendly_deepblue] \arrow[rd,color=friendly_deepblue] \arrow[ll, "\tau"', dashed] &                                                           & \textcolor{friendly_deepblue}{(1,3)} \arrow[rd,color=friendly_deepblue] \arrow[ru,color=friendly_deepblue] \arrow[ll, "\tau"', dashed,color=friendly_deepblue] &                                                           & \textcolor{friendly_deepblue}{(2,3)} \arrow[ll, "\tau"', dashed,color=friendly_deepblue] \arrow[ru] \arrow[rd] &                                                           & \textcolor{friendly_beige}{{(3,3)}} \arrow[ru,color=friendly_beige] \arrow[rd,color=friendly_beige] \arrow[ll, "\tau"', dashed] &                                                           & \textcolor{friendly_beige}{(4,3)} \arrow[ru,color=friendly_beige] \arrow[rd,color=friendly_beige] \arrow[ll, "\tau"', dashed,color=friendly_beige] &                                                           & \textcolor{friendly_beige}{(5,3)} \arrow[ru] \arrow[rd] \arrow[ll, "\tau"', dashed,color=friendly_beige] &                                                           & {\textcolor{friendly_deepblue}{{(0,3)}}} \arrow[ll, "\tau"', dashed]  \\
                  & \textcolor{friendly_beige}{(5,2)} \arrow[rd] \arrow[ru,color=friendly_beige] &                                                & \textcolor{friendly_deepblue}{{(0,2)}} \arrow[ru,color=friendly_deepblue] \arrow[rd,color=friendly_deepblue] \arrow[ll, "\tau"', dashed] &                                                           & \textcolor{friendly_deepblue}{(1,2)} \arrow[ru,color=friendly_deepblue] \arrow[rd,color=friendly_deepblue] \arrow[ll, "\tau"', dashed,color=friendly_deepblue] &                                                           & \textcolor{friendly_deepblue}{(2,2)} \arrow[rd,color=friendly_deepblue] \arrow[ru,color=friendly_deepblue] \arrow[ll, "\tau"', dashed,color=friendly_deepblue] &                                                           & \textcolor{friendly_beige}{{(3,2)}} \arrow[ru,color=friendly_beige] \arrow[rd,color=friendly_beige] \arrow[ll, "\tau"', dashed] &                                                           & \textcolor{friendly_beige}{(4,2)} \arrow[ru,color=friendly_beige] \arrow[rd,color=friendly_beige] \arrow[ll, "\tau"', dashed,color=friendly_beige] &                                                           & \textcolor{friendly_beige}{(5,2)} \arrow[ru] \arrow[rd] \arrow[ll, "\tau"', dashed,color=friendly_beige] &                                                           & {\textcolor{friendly_deepblue}{{(0,2)}}} \arrow[ru,color=friendly_deepblue] \arrow[ll, "\tau"', dashed]             &                                    \\
\textcolor{friendly_beige}{(5,1)} \arrow[ru,color=friendly_beige] &                              & \textcolor{friendly_deepblue}{{(0,1)}} \arrow[ru,color=friendly_deepblue] \arrow[ll, "\tau"', dashed] &                                                           & \textcolor{friendly_deepblue}{(1,1)} \arrow[ru,color=friendly_deepblue] \arrow[ll, "\tau"', dashed,color=friendly_deepblue]            &                                                           & \textcolor{friendly_deepblue}{(2,1)} \arrow[ru,color=friendly_deepblue] \arrow[ll, "\tau"', dashed,color=friendly_deepblue]            &                                                           & \textcolor{friendly_deepblue}{(3,1)} \arrow[ll, "\tau"', dashed,color=friendly_deepblue] \arrow[ru]            &                                                           & \textcolor{friendly_beige}{{(4,1)}} \arrow[ru,color=friendly_beige] \arrow[ll, "\tau"',dashed]            &                                                           & \textcolor{friendly_beige}{(5,1)} \arrow[ru,color=friendly_beige] \arrow[ll, "\tau"', dashed,color=friendly_beige]            &                                                           & {\textcolor{friendly_deepblue}{{(0,1)}}} \arrow[ru,color=friendly_deepblue] \arrow[ll, "\tau"', dashed]                        &                                                          &                                   
\end{tikzcd}
\]
where vertices with equal labels are identified. Let $\sigma$ be the automorphism of $\Z T$ given by reflection about the line $\{(i,3) \, \vert \, i \in \Z\}$. We also denote the induced automorphism of $\Z T/\langle \tau^{-6}\rangle$ by $\sigma$. Then, $\sigma\tau^{-3}$ sends the \textcolor{friendly_deepblue}{blue} subquiver above to the \textcolor{friendly_beige}{orange} subquiver, and \textit{vice versa}. One sees that the $\sigma\tau^{-3}$-orbits of vertices, arrows, and meshes all have cardinality equal to the order of $\sigma$ (namely 2), as shown in the proof of \Cref{lem:Self-injective_eu}. 
One can compute directly that
\[\chi_{\mathrm{AR}}(\Z T/\langle \sigma\tau^{-3}\rangle) = 15 -24+15=6, \]
which is consistent with \Cref{lem:Self-injective_eu}.
\end{example}
We are ready to prove the main theorem of this section.

\begin{theorem}\label{thm:Self-injective}
    Let $\Lambda$ be a connected non-semisimple representation-finite self-injective bound path $k$-algebra of rank $n$.
    Then
    \begin{align*}
   \MagnitudeOfModuleCat{\Lambda} = n\left(2{h_{T}-1\over |{T}_0| }-1\right),
    \end{align*}
  where $T=(T_0,T_1)$ is a Dynkin quiver such that the stable Auslander--Reiten quiver of $\leftmod{\Lambda}$ is an admissible quotient of $\Z T$, and $h_{T}$ is the Coxeter number of $T$.
In particular, since $h_{T}-1\geq |{T}_0|$, we have $\MagnitudeOfModuleCat{\Lambda}\geq n$. 
\end{theorem}
\begin{proof}
    Every non-projective vertex in the Auslander--Reiten quiver of $\leftmod{\Lambda}$ is stable. Since $\Lambda$ is non-semisimple, the stable Auslander--Reiten quiver of $\leftmod{\Lambda}$ is non-empty. By \Cref{thm:Rie80}\eqref{thm:Rie80sAR}, there exists a Dynkin quiver $T$ with the desired property. Also, by \Cref{rem:sAR}, the stable Auslander--Reiten quiver of $\leftmod{\Lambda}$ has $n$ fewer vertices than the Auslander--Reiten quiver of $\leftmod{\Lambda}$. Furthermore, by \Cref{lem:DK}\eqref{lem:DK_arrows}, it has $2n$ fewer arrows. It follows that 
  \begin{align*}
   \MagnitudeOfModuleCat{\Lambda}  &= \chi_{\mathrm{AR}}(\Z {T}/\langle \sigma\tau^{-r} \rangle) - n & (\text{\Cref{thm:Rie80}\eqref{thm:Rie80sAR}}) \\
   &= 2r - n. & (\text{\Cref{lem:Self-injective_eu}})
   \end{align*}
   Finally, it is known that $r= {n(h_{T}-1)\over |{T}_0|}$ \cite[2.3]{BLR81}, which can be substituted into the expression above.
  \end{proof}

Note that the orientation of $T$ in \Cref{thm:Self-injective} can be chosen freely. Indeed, up to isomorphism of translation quivers, the repetitive quiver $\Z T$ does not depend on this orientation, so the admissible group action can be transferred along such an isomorphism.

We now deduce that \Cref{conjecture} holds for representation-finite self-injective bound path $k$-algebras (note that $\MagnitudeOfModuleCat{\Lambda}\geq n$ by \Cref{thm:Self-injective}).

\begin{corollary}\label{cor:Self-injective}
In the setting of \Cref{thm:Self-injective}, the following assertions are equivalent.
	\begin{enumerate}
        	\item\label{thm:Self-injective_specialbiserial_c} $\Lambda$ is special biserial.
         \item\label{thm:Self-injective_specialbiserial_a} $\MagnitudeOfModuleCat{\Lambda}=n$.
          \item\label{thm:Self-injective_specialbiserial_b} ${T}$ is of type $A_{|T_0|}$.
        \end{enumerate}
\end{corollary}
\begin{proof}
The fact that \eqref{thm:Self-injective_specialbiserial_c} implies \eqref{thm:Self-injective_specialbiserial_a} was shown in \Cref{thm:SpecialBiserial}. Next, if \eqref{thm:Self-injective_specialbiserial_a} holds, then $n\left(2{h_{T}-1 \over |{T}_0|}-1\right)=n$ by \Cref{thm:Self-injective}, which leads to $h_{T}-1=|{T}_0|$. This can only happen if ${T}$ is of type $A_{|T_0|}$, which is exactly the assertion in \eqref{thm:Self-injective_specialbiserial_b}. Finally, if ${T}$ is of type $A_{|T_0|}$, \Cref{rem:sAR} and \Cref{lem:WW}\eqref{lem:ARseq} imply that $\Lambda$ is special biserial.
\end{proof}

\begin{remark}
Suppose that $k$ is algebraically closed and of prime characteristic $p$. Let $G$ be a finite group such that $p$ divides the order of $G$. It is well-known that a block $B$ of $kG$ is self-injective. If $B$ is representation-finite, it follows from the proof of \Cref{cor:block} that $\MagnitudeOfModuleCat{B}=n$, where $n$ denotes the number of simple $B$-modules up to isomorphism. Under the additional assumption that $B$ is connected and not semisimple, \Cref{cor:Self-injective} now lets us describe the stable Auslander--Reiten quiver of $\leftmod{B}$ as a non-trivial admissible quotient of $\Z T$, where $T$ is a Dynkin quiver of type $A$.
\end{remark}

\bibliographystyle{ourIEEEstyle.bst}
\bibliography{magrefs}

\end{document}